\documentclass[11pt,letterpaper]{amsart}

\usepackage{subfiles}
\usepackage{xr}

\usepackage{amsmath,amssymb,amsfonts,amsthm}
\usepackage{color}
\usepackage[colorlinks=true,linkcolor=blue,citecolor=blue,urlcolor=blue,pdfborder={0 0 0}]{hyperref}
\usepackage{cleveref}
\usepackage{caption}
\usepackage{thmtools}
\usepackage{mathrsfs}

\usepackage[textsize=tiny]{todonotes}

\crefname{theorem}{Theorem}{Theorems}
\crefname{thm}{Theorem}{Theorems}
\crefname{lemma}{Lemma}{Lemmas}
\crefname{claim}{Claim}{Claims}
\crefname{lem}{Lemma}{Lemmas}
\crefname{remark}{Remark}{Remarks}
\crefname{prop}{Proposition}{Propositions}
\crefname{proposition}{Proposition}{Propositions}
\crefname{defn}{Definition}{Definitions}
\crefname{definition}{Definition}{Definitions}
\crefname{corollary}{Corollary}{Corollaries}
\crefname{conjecture}{Conjecture}{Conjectures}
\crefname{question}{Question}{Questions}
\crefname{chapter}{Chapter}{Chapters}
\crefname{section}{Section}{Sections}
\crefname{part}{Part}{Parts}
\crefname{figure}{Figure}{Figures}

\theoremstyle{plain}
\newtheorem{thm}{Theorem}[section]

\newtheorem{lemma}[thm]{Lemma}

\newtheorem{lem}[thm]{Lemma}

\newtheorem{cor}[thm]{Corollary}

\newtheorem{prop}[thm]{Proposition}

\theoremstyle{definition}
\newtheorem{defn}[thm]{Definition}

\newtheorem{example}[thm]{Example}

\theoremstyle{remark}
\newtheorem{remark}[thm]{Remark}
\newtheorem{rk}[thm]{Remark}
\newtheorem*{remark*}{Remark}

\numberwithin{equation}{section}

\usepackage{extarrows}
\usepackage{commath}
\usepackage{mathtools}
\usepackage{bbm}
\usepackage{enumitem}
\usepackage{cite}

\renewcommand{\P}{\mathbb P}
\newcommand{\E}{\mathbb E}

\newcommand{\R}{\mathbb R}
\newcommand{\Z}{\mathbb Z}

\newcommand{\cA}{\mathcal A}
\newcommand{\cB}{\mathcal B}

\newcommand{\cE}{\mathcal E}
\newcommand{\cF}{\mathcal F}

\newcommand{\cH}{\mathcal H}

\newcommand{\cL}{\mathcal L}

\newcommand{\sC}{\mathscr C}

\newcommand{\Aut}{\operatorname{Aut}}

\def\P{\mathbb{P}}

\newcommand{\bra}[1]{\left[#1\right]}

\newcommand{\seq}[1]{\left(#1\right)}

\DeclareMathSymbol{\leqslant}{\mathalpha}{AMSa}{"36} 
\DeclareMathSymbol{\geqslant}{\mathalpha}{AMSa}{"3E} 
\DeclareMathSymbol{\eset}{\mathalpha}{AMSb}{"3F}     

\renewcommand{\epsilon}{\varepsilon}

\newcommand{\p}{\mathbb{P}}
\newcommand{\e}{\mathbb{E}}
\renewcommand{\bra}[1]{\left(#1\right)}
\newcommand{\sqbra}[1]{\left[#1\right]}
\newcommand{\cubra}[1]{\left\{#1\right\}}
\newcommand{\den}[1]{\left\lVert#1\right\rVert}
\renewcommand{\seq}{_{n \geq 1}}
\newcommand{\eeq}[2]{\begin{equation} \label{#1} #2 \end{equation}}
\newcommand{\spliteq}[1]{ \[ \begin{split} #1 \end{split} \]}
 
\newcommand{\act}[0]{\operatorname{Act}}

\renewcommand{\abs}[1]{\left\lvert#1\right\rvert}

\newcommand{\eps}{\varepsilon}

\title[Uniqueness of the giant component
]{ Supercritical percolation on finite transitive graphs I: \\Uniqueness of the giant component
}

\author{Philip Easo}
\thanks{Email: \href{mailto:peaso@caltech.edu}{peaso@caltech.edu} and \href{mailto:t.hutchcroft@caltech.edu}{t.hutchcroft@caltech.edu} \\ The Division of Physics, Mathematics and Astronomy, California Institute of Technology.}

\author{Tom Hutchcroft}

\date{\today}

\begin{document}

\begin{abstract}
Let $(G_n)\seq = ((V_n,E_n))\seq$ be a sequence of finite, connected, vertex-transitive graphs with volume tending to infinity. We say that a sequence of parameters $(p_n)\seq$ in $[0,1]$ is \emph{supercritical} with respect to Bernoulli bond percolation $\p_p^G$ if there exists $\varepsilon >0$ and $N<\infty$ such that
\[
	\p_{(1-\varepsilon)p_n}^{G_n} \bra{ \text{the largest cluster contains at least  $\varepsilon |V_n|$ vertices}} \geq \varepsilon
\]
for every $n\geq N$ with $p_n <1$.
We prove that if $(G_n)\seq$ is sparse, meaning that the degrees are sublinear in the number of vertices, then the supercritical giant cluster is unique with high probability in the sense that if $(p_n)\seq$ is supercritical then 
\[
	\lim_{n\to\infty}\p_{p_n}^{G_n} \bra{ \text{the second largest cluster contains at least $c|V_n|$ vertices} } = 0
\]
for every $c>0$.
This result is new even under the stronger hypothesis that $(G_n)\seq$ has uniformly bounded vertex degrees, in which case it verifies a conjecture of Benjamini (2001). Previous work of many authors had established the same theorem for complete graphs, tori, hypercubes, and bounded degree expander graphs, each using methods that are highly specific to the examples they treated. We also give a complete solution to the problem of supercritical uniqueness for \emph{dense} vertex-transitive graphs, establishing a simple  necessary and sufficient isoperimetric condition for uniqueness to hold.
\end{abstract}

\maketitle

\tableofcontents

\section{Introduction} \label{section:introduction}

Let $G=(V,E)$ be a countable graph that is connected and vertex-transitive, meaning that for all vertices $u,v \in V$  there is a graph automorphism $\phi \in \operatorname{Aut} G$ with $\phi(u)=v$. Given $p \in [0,1]$, \emph{Bernoulli bond percolation} $\p_p^G$ (abbreviated $\p_p$ when the choice of $G$ is clear from context) is the distribution of a random spanning subgraph $\omega$ formed by independently including each edge with probability $p$. We identify $\omega$ with an element of $\{  0,1 \}^E$ where $\omega(e) = 1$ means that the edge $e$ is present in $\omega$. The edges in $\omega$ are called \emph{open} and the rest are called \emph{closed}. We are interested primarily in the geometry of the connected components of $\omega$, which we refer to as \emph{clusters}.
Much of the interest in the model stems from the existence of a \emph{phase transition}: For infinite graphs, there is typically\footnote{We keep the meaning of `typically' intentionally vague. One very general conjecture \cite[Question 2]{beyond-euclid} is that $p_c<1$ for all (not necessarily transitive) infinite graphs with isoperimetric dimensional strictly greater than $1$.} a \emph{critical probability} $0<p_c<1$ such that every cluster is finite almost surely when $p< p_c$, while at least one infinite cluster exists almost surely when $p>p_c$. For large finite graphs, one typically observes a similar phase transition in which a \emph{giant} component, containing a positive proportion of all vertices, emerges as $p$ is varied through a small interval. The regime in which an infinite/giant cluster exists is known as the \emph{supercritical phase}. It is now known that the supercritical phase is always non-degenerate for bounded degree transitive graphs that are strictly more than one-dimensional in an appropriate coarse-geometric sense: this was proven for infinite graphs by Duminil-Copin, Goswami, Raoufi, Severo, and Yadin \cite{non-triviality} and for finite graphs by the second author and Tointon~\cite{hutchcroft+tointon}.

\medskip

Once one knows that the supercritical phase is non-degenerate, so that infinite/giant clusters exist for sufficiently large values of $p$, a central problem is to understand the \emph{number} of these clusters. 
For \emph{infinite} transitive graphs, this is the subject of a famous conjecture of Benjamini and Schramm~\cite{beyond-euclid} stating that the infinite cluster is unique for every $p>p_c$ if and only if the graph is \emph{amenable}. The `if' direction of this conjecture follows from the classical work of Aizenman, Kesten, and Newman~\cite{AKN} and Burton and Keane~\cite{burton-keane}, while the `only if' direction remains open in general; see e.g.\ \cite{MR4155221,MR3962879,MR3616205,MR1756965,MR3352259} for an overview of what is known.

\medskip

In contrast, for \emph{finite} transitive graphs, Benjamini \cite[Conjecture 1.2]{benjamini-unpublished} conjectured in 2001 that the giant cluster should \emph{always} be unique in the supercritical regime, irrespective of the geometry of the graph. Several works, some of which are very classical, have established versions of this conjecture in special cases including for complete graphs \cite{MR148055,MR756039}, hypercubes \cite{MR671140,MR1139488}, Euclidean tori of fixed dimension \cite{MR2213947}, and bounded degree expanders \cite{percolation-expanders}. Each of these works uses methods that are very specific to the example it treats, with the analysis of the tori $(\mathbb{Z}^d/n\mathbb{Z}^d)\seq$ in dimension $d\geq 3$ relying in particular on the important and technically challenging work of Grimmett and Marstrand \cite{grimmett-marstrand}. A related conjecture giving mild geometric conditions under which it should be impossible to have multiple giant clusters \emph{both above and at criticality} was subsequently stated in the influential work of Alon, Benjamini, and Stacey \cite[Conjecture 1.1]{percolation-expanders}.

\medskip

A central difficulty in the study of this conjecture, and in the study of percolation on finite graphs more generally, is that many of the most important qualitative tools used to study the infinite case, such as the ergodic theorem,  break down completely in the finite case. For example, adapting the Burton--Keane uniqueness proof to the case of finite graphs merely shows that each vertex is unlikely to have three distinct large clusters in a small vicinity around it -- a statement that need not in general be in tension with the existence of multiple giant components in a finite graph. Indeed, while the Burton--Keane proof applies to arbitrary \emph{insertion-tolerant} automorphism-invariant percolation processes, it is possible to construct insertion-tolerant automorphism-invariant percolation processes on the torus $(\Z/n\Z)^2$ that have multiple giant components with high probability.  In fact, for the highly asymmetric torus $(\Z/n \Z)\times (\Z/ 2^n \Z)$ one even has multiple giant clusters with good probability for Bernoulli percolation at appropriate values of $p$, although this arises as a feature of a discontinuous phase transition rather than of the supercritical phase \emph{per se}. See \cref{subsec:counterexamples} for further discussion of both examples. In light of these difficulties, any treatment of the uniqueness problem for finite transitive graphs must involve new techniques and use finer properties of supercritical percolation than in the infinite case.

\medskip

In this paper we resolve Benjamini's conjecture and hence also the supercritical case of the Alon--Benjamini--Stacey conjecture, giving a complete solution to the problem of supercritical uniqueness on large, finite, vertex-transitive graphs. In the forthcoming work in this series, we will prove moreover that the density of the giant component is concentrated, local, and equicontinuous in the supercritical regime and prove analogous theorems for the Fortuin-Kasteleyn random cluster model, Ising model, and Potts model.

\begin{defn} \label{defn:supercritical-percolation}

We will assume all graphs to be locally finite and to contain at least one vertex. Let $\cF$ be the set of all isomorphism classes of finite, connected, simple (i.e., not containing loops or multiple edges), vertex-transitive graphs. (We will usually suppress the distinction between graphs and their isomorphism classes as much as possible when this does not cause any confusion.) Given an infinite set $\cH\subseteq \cF$, a function $\phi:\cH\to \R$, and $\alpha \in \R$ we write $\lim_{G\in \cH}\phi(G)=\alpha$ or ``$\phi(G)\to \alpha$ as $G\to\infty$ in $\cH$'' to mean that for each $\eps>0$ there exists $N$ such that $|\phi(G)-\alpha|\leq \eps$ for every $G\in \cH$ with at least $N$ vertices, or equivalently that $\phi(G_n)\to \alpha$ for some (and hence every) enumeration $\cH=\{G_1,G_2,\ldots\}$ of $\cH$. Similar conventions apply to the definition of $\limsup_{G\in \cH}$, $\liminf_{G\in \cH}$, and limits that may be equal to $+\infty$ or $-\infty$.

\medskip


Let $G = (V,E)$ be a countable graph and consider a percolation configuration $\omega \in \{ 0,1 \}^E$. The connected components of $\omega$ are called \emph{clusters}. We write $K_u$ to denote the cluster containing the vertex $u$ and write $u \leftrightarrow v$ for the event that $K_u = K_v$. Given a subset $W$ of $V$, the \emph{volume} of $W$ is the number of vertices in $W$, denoted $\abs{W}$, while if $G$ is finite, the \emph{density} of $W$ is defined to be the ratio $\den{W} := \abs{W}/\abs{V}$. We write $K_1,K_2,\ldots$ for the clusters of $\omega$ in decreasing order of volume. (Note the slight abuse of notation: here $K_1$ does not mean the cluster of a vertex labelled `1'.)

\medskip

Given an infinite set $\cH \subseteq \cF$, we say that an assignment of parameters $p_c:\mathcal H \to [0,1]$ is a \textbf{percolation threshold} if
\begin{enumerate}
  \item
 $\lim_{G \in \mathcal H} \p_{(1-\eps)p_c}^{G} \bra{ \den{K_1} \geq c } = 0$ for every $\eps,c>0$, and
  \item For every $\eps>0$ there exists $\alpha>0$ such that \[\lim_{G \in \mathcal H} \p_{(1+\eps)p_c}^G \bra{ \den{K_1} \geq \alpha } = 1,\] where we set $\p_p^G=\p_1^G$ for $p\geq 1$.
 \end{enumerate}
Note that critical thresholds are \emph{not} unique (when they exist), but any two percolation thresholds $p_c,\tilde p_c:\cH\to [0,1]$ must satisfy
$p_c(G)\sim \tilde p_c(G)$ 
as $G\to\infty$ in $\cH$. When a percolation threshold $p_c:\cH\to[0,1]$ exists, we say that $p:\mathcal H \to [0,1]$ is \emph{supercritical} if $\mathcal H' := \{ G \in \mathcal H : p(G) < 1 \}$ is finite or if $\mathcal H'$ is infinite and satisfies
\[\liminf_{G \in \mathcal H'} \frac{p(G)}{p_c(G)} > 1.\] We generalise this definition to include the case that $p_c$ does not exist by saying that $p$ is \textbf{supercritical} if $\mathcal H' := \{ G \in \mathcal H : p(G) < 1 \}$ is finite or if $\mathcal H'$ is infinite and there exists $\eps > 0$ such that 
\[
\liminf_{G \in \mathcal H'} \p_{(1-\eps)p}^G \bra{ \den{K_1} \geq \eps } \geq \eps.
\]
Note that these two definitions of supercriticality coincide when $\cH$ admits a threshold function, and in particular that the definition of supercriticality does not depend on the choice of threshold function. (Without the $(1-\eps)$ factor in $\p_{(1-\eps)p}^G$, these definitions would not always coincide, for example for the highly asymmetric torus discussed in Example 5.1, which has giant clusters with good probability at a percolation threshold bounded away from 1.) The reason for introducing the set $\mathcal H'$ is to ensure that every family has a supercritical sequence of parameters, namely the constant assignment $p(G):=1$ for all $G$. It will also be helpful to have the finitary version of this definition: Given a \emph{single} finite, connected, simple, vertex-transitive graph $G=(V,E)$ and given any $\eps > 0$, we say that a parameter $p \in [0,1]$ is \textbf{$\eps$-supercritical} for $G$ if $\p_{(1-\eps)p}^{G}(\den{K_1} \geq \eps) \geq \eps$ and $\abs{V} \geq 2 \eps^{-3}$. (There is some flexibility in how to choose this latter, technical condition that $\abs{V}$ is not too small.)

\medskip

We say that $\mathcal H$ has the \textbf{supercritical uniqueness property} if 
\[\lim_{G \in \mathcal H} \p_p^G \bra{ \den{K_2} \geq \eps } = 0\] for every supercritical $p:\cH\to[0,1]$ and every constant $\eps > 0$.

\end{defn}

We begin by stating our main result in the simplest-to-state and most interesting case when we have an infinite set $\mathcal H \subseteq \mathcal F$ that is \emph{sparse}, meaning that the average vertex degree $d(G) :=2|E(G)|/|V(G)|$ of $G \in \cH$ (which is the exact degree of every vertex because $G$ is regular) satisfies $d(G)=o(|V(G)|)$ as $G \to \infty$ in $\mathcal H$. Note that in particular, if $\mathcal H$ has uniformly bounded vertex degrees (i.e.\! $\sup_{G \in \mathcal H} d(G) < \infty$) then $\mathcal H$ is sparse.

\begin{thm} \label{thm:main_sparse}
	Let $\cH \subseteq \cF$ be an infinite set. If $\cH$ is sparse, then $\mathcal H$ has the supercritical uniqueness property.
\end{thm}

\begin{rk}
The restriction to simple graphs is not very important and could be replaced by e.g.\ the assumption that there are a bounded number of parallel edges between any two vertices.
\end{rk}

\begin{rk}
The Alon--Benjamini--Stacey conjecture \cite[Conjecture 1.1]{percolation-expanders} would follow immediately from \cref{thm:main_sparse} together with the plausible claim that the percolation phase transition is always continuous for bounded degree graph families satisfying the $\operatorname{diam}(G)=o(|V(G)|/\log |V(G)|)$ condition they consider. More formally, such a claim would state that if $\mathcal H \subseteq \mathcal F$ is an infinite set with uniformly bounded vertex degrees that satisfies this condition, and $p : \mathcal H \to [0,1]$ is any assignment of parameters such that $\liminf_{G \in \mathcal H}\p_{p(G)}^{G}(\|K_1\|\geq c)>0$ for some $c>0$, then $p$ is supercritical in our sense. Unfortunately such a claim seems to be completely beyond the scope of present techniques and is a major open problem even for e.g.\ the three-dimensional torus $(\Z/n\Z)^3$. Indeed, it appears to be an open problem to prove that there are not multiple giant components at criticality in this example.
\end{rk}

All the proofs in our paper are effective in the sense that they can in principle be used to produce explicit bounds on, say, the expected density of the second largest cluster. While we have not kept track of what these bounds are in all cases, we make note of the following simple explicit estimate implying \cref{thm:main_sparse} in the case that the graphs in question have bounded or subalgebraic vertex degrees.

\begin{thm} \label{thm:quantitative-sup}
There exists a universal constant $C$ such that if $G = (V,E)$ is a finite, simple, connected, vertex-transitive graph with vertex degrees bounded by $d$, and $p \in [0,1]$ is $\eps$-supercritical for $G$ for some $\eps>0$, then 
\[
	\p_{p}^{G} \left(\|K_2\| \geq \lambda  e^{ C\eps^{-18}} \sqrt{\frac{\log d}{\log |V|}} \right) \leq \frac{1}{\lambda}
\]
for every $\lambda \geq 1$. 
\end{thm}

Note that this bound is only useful under the subalgebraic degree condition $\log d \ll \log |V|$. The constant given by our proof is fairly large, of order around $10^6$. 

\medskip

This bound has not been optimized and is known to be very far from optimal in classical examples. Indeed, the second largest cluster in supercritical percolation is known to be of order $\Theta(\log |V|)$ with high probability on both the complete graph \cite{MR148055,MR756039} and the hypercube \cite{MR1139488}, while for a Euclidean torus of fixed dimension $d$, it is of order $\Theta\bigl((\log |V|)^{d/(d-1)}\bigr)$ \cite{MR2213947}. Similar results are established for a large class of \emph{dense} graphs in \cite{dense}. 

\begin{rk}

Let us now explain the relationship between our theorem and the Benjamini--Schramm $p_c<p_u$ conjecture \cite{beyond-euclid}. Suppose $(G_n)\seq$ is a bounded degree expander sequence converging locally to some infinite nonamenable transitive graph $G$. One may deduce either from our results or those of \cite{percolation-expanders} (see also \cite{MR4275958}) that there is always a unique giant component with high probability for supercritical percolation on $G_n$, a result that seems to be in tension with the conjectured existence of a non-uniqueness phase for percolation on the limit graph $G$.
Naively, one might think that our definition of supercriticality for finite graphs should therefore be thought of more properly as an analogue of the \emph{uniqueness phase} ($p>p_u$) for infinite graphs.

This is misleading. Indeed, it was proven in \cite{MR2773031} that if $(G_n)\seq$ is a sequence of transitive, bounded degree expanders converging to an infinite, transitive, nonamenable graph $G$, then a sequence $(p_n)\seq$ is supercritical if and only if $\liminf p_n > p_c(G)$. The uniqueness/non-uniqueness transition on the limit graph $G$ does manifest itself in the approximating finite graphs $G_n$, but as a transition in the \emph{metric distortion} of the giant component rather than its uniqueness: the length of the path connecting two neighbouring vertices of $G_n$ given that both vertices belong to the giant is tight as $n\to\infty$ when $p>p_u(G)$ and is not tight when $p_c(G) < p < p_u(G)$. In the second case, the open path connecting two such vertices in $G_n$ has good probability to be very long, thus the two vertices become disconnected with positive probability in the limit. See  \cite{MR2384409} for related results and open problems for hypercube percolation.

\end{rk}

\textbf{The dense case.}
We now discuss how our results extend to \emph{dense} graphs, where vertices have degree proportional to the number of vertices. In contrast to \cref{thm:main_sparse}, it is not true in general that dense graph families have the supercritical uniqueness property. Suppose for example that $G_n$ is the Cartesian product of the complete graphs $K_2$ and $K_n$ and that $p_n = 2/n$ for every $n\geq 1$. Then we have by the classical theory of Erd\H{o}s--R\'enyi random graphs that each copy of $K_n$ contains a giant component with high probability, while the number of `horizontal' edges connecting the two copies of $K_n$ converges in distribution to a Poisson$(2)$ random variable and is therefore equal to zero with probability bounded away from zero as $n\to\infty$. Thus, the number of giant clusters in this example is unconcentrated and can be equal to either one or two, each with good probability.

\medskip

Our next main result shows that examples of roughly this form are the only transitive counterexamples to the supercritical uniqueness property. 

\medskip

We will in fact characterise the failure of the supercritical uniqueness property for dense vertex-transitive graphs in two equivalent ways. We say that an infinite set $\cH \subseteq \cF$ is \emph{dense} if $\liminf_{G \in \cH}|E(G)|/|V(G)|^2>0$; this is equivalent to the vertex degree $d(G)$ growing linearly in the number of vertices in the sense that $\liminf_{n\to\infty} d(G)/|V(G)| >0$.

\begin{defn}
Let $\mathcal H \subseteq F$ be a set. Given $m \in \{2,3,\dots \}$, we say that $\cH$ is \emph{$m$-molecular} if it is infinite and dense and there exists a constant $C < \infty$ such that for each $G \in \mathcal H$ there exists a set of edges $F \subseteq E(G)$ satisfying the following conditions:
\begin{enumerate}
		\item $G\setminus F$ has $m$ connected components;
		\item $F$ is invariant under the action of $\operatorname{Aut}G$;
		\item $\abs{F} \leq C \abs{V(G)}$.
			\end{enumerate}
These conditions imply that the $m$ connected components of $G \setminus F$ are dense, vertex-transitive, and isomorphic to each other. For example, the family of Cartesian products $\{K_n \square K_2 : n \geq 1\}$ discussed above is $2$-molecular.
	We say $\mathcal H$ is \emph{molecular} if it is $m$-molecular for some $m\geq 2$.
\end{defn}

\begin{defn}
Let $G=(V,E)$ be a finite graph. For each set $A \subseteq V$, we write $\partial_E A$ for the set of edges that have one endpoint in $A$ and the other in $V\setminus A$. For each $\theta \in (0,1/2]$, the quantity $\textsc{Separator}(G,\theta)$ is defined to be
\begin{multline*}
\textsc{Separator}(G,\theta):=\\\min\left\{ |\partial_E A| : \theta \sum_{v\in V} \deg(v) \leq \sum_{v\in A} \deg(v) \leq (1-\theta) \sum_{v\in V} \deg(v)\right\}.
\end{multline*}
In other words, $\textsc{Separator}(G,\theta)$ is the minimal number of edges needed to cut $G$ into two pieces of roughly equal size. We say that a set $\mathcal H$ of isomorphism classes of finite, simple graphs has \emph{linear $\theta$-separators} if $\mathcal H$ is infinite and $\limsup_{G \in \mathcal H}\textsc{Separator}(G,\theta)/|V(G)|<\infty$.
\end{defn}

The following theorem provides a complete solution to the problem of supercritical uniqueness for Bernoulli bond percolation on finite vertex-transitive graphs and implies \cref{thm:main_sparse} as a special case.

\begin{thm} \label{thm:main}
For every infinite set $\mathcal H \subseteq \mathcal F$, the following are equivalent:
\begin{enumerate}
	\item[\emph{(i)}] $\mathcal H$ does not have the supercritical uniqueness property;
	\item[\emph{(ii)}] $\mathcal H$ contains a subset that is molecular;	
	\item[\emph{(iii)}] $\mathcal H$ contains a subset with linear $1/3$-separators;
	\item[\emph{(iv)}] $\mathcal H$ contains a dense subset with linear $\theta$-separators for some $\theta \in (0,1/2]$.
\end{enumerate}
\end{thm}

 The dense case of this result sharpens the transitive case of a theorem of Bollob\'as, Borgs, Chayes, and Riordan \cite{dense}, who proved supercritical uniqueness for any (not necessarily transitive) dense graph sequence converging to an irreducible graphon. In our language, their result states that a dense graph family has the supercritical uniqueness property whenever it does not have any subquadratic separators, i.e., whenever \[\liminf_{G \in \mathcal H}\textsc{Separator}(G,\theta)/|V(G)|^2 > 0\] for every $\theta \in (0,1/2]$ (see \cite[Lemma 7]{dense}). In fact, since they also prove that the giant cluster density is zero at the percolation threshold, their results imply uniqueness of the giant cluster for \emph{all} (not necessarily supercritical) assignments of parameters. The same authors also established a formula for the limiting critical probability of dense graph sequences that we will use to prove the implication (iv) $\Rightarrow$ (i) of \cref{thm:main}.  Further comparison of our results with those of \cite{dense} is given in \cref{rem:dense}.

\begin{rk}
  In \cite{easo2022existence} the first author has built on the results of the present paper to characterise which infinite subsets of $\mathcal F$ admit a percolation threshold. The obstacle to having a percolation threshold turns out to be the presence of molecular subsets for \emph{infinitely many} values of $m \in \{2,3,\ldots\}$. In particular, every infinite subset of $\mathcal F$ that is sparse admits a percolation threshold. So a posteriori, for \cref{thm:main_sparse} it suffices to work with the original (more natural) definition of supercritical assignments of parameters, which refers to the percolation threshold, rather than the more general definition. Note the surprising logical order here: we first proved uniqueness of the supercritical giant cluster (in the present paper), then this was used to prove that there exists a percolation threshold for the emergence of a giant cluster (in \cite{easo2022existence}). 
\end{rk}

\subsection{About the proof and organization}

We now briefly overview the structure of the paper and outline the proofs of the main steps.

\medskip

\noindent\textbf{Section 2: Lower bounds on point-to-point connection probabilities.}
In this section we prove that for $\eps$-supercritical percolation on any finite, simple, connected, vertex-transitive graph, we always have a uniform lower bound $\p_p(x\leftrightarrow y) \geq \delta(\eps)>0$ on the probability that any two given vertices are connected. This was previously known only in the bounded degree case, with constants depending on the degree. Our argument starts by partitioning the vertex set into classes within which we have such a lower bound then recursively merges these classes until a single class contains the entire graph. The merging step makes use of a new high-degree version of insertion-tolerance, which allows us to open a single edge in a sufficiently large random set of edges.

\medskip

\noindent\textbf{Section 3: Uniqueness under the sharp density property.}
In this section we prove that the supercritical uniqueness property holds for any infinite set $\mathcal H \subseteq \mathcal F$ satisfying the \emph{sharp density property}, meaning that $\p_p(\|K_1\|\geq \alpha)$ has a sharp threshold for every $\alpha \in (0,1]$ in an appropriately uniform sense. This section is at the heart of the paper and contains the most significant new arguments. 

\medskip

The proof has two parts. First, given any particular supercritical parameter $p$, we use the sharp density property to (non-constructively) deduce the existence of a smaller parameter $q \leq p$ such that under $\p_q$ there is a giant cluster whose density is \emph{concentrated}, i.e., lies in a small interval with high probability. The point-to-point connection probability lower bound easily implies that this giant cluster is the \emph{unique} giant cluster under $\p_q$ with high probability. 

\medskip

The remainder of the proof consists in showing that non-uniqueness of the giant cluster under $\p_p$ would imply non-concentration of the density of the giant cluster under $\p_q$, establishing a contradiction. To this end we introduce a new object called a \emph{sandcastle}, which is a large subgraph that is not resilient to $q/p$-bond percolation. We observe that under the hypothesis that there are at least two giant clusters under $\p_p$ with good probability, at least one of these clusters must be a sandcastle with good probability. As we pass from $\p_p$ to $\p_q$ in the standard monotone coupling of these measures, this sandcastle-cluster disintegrates into small clusters with good probability by definition. On this event, the giant cluster under $\p_q$ is constrained to the complement of this sandcastle-cluster from $\p_p$, where edges are distributed (conditionally) independently. We argue that if this is the case then there must be a subset of $V$ with density significantly less than $1$ that has good probability under $\p_q$ to contain a cluster of approximately the same size as the typical global size of the largest cluster in $V$. Finally, we use the existence of this subset together with Harris' inequality and the point-to-point lower bound to deduce that $\|K_1\|$ is abnormally large with good probability under $\p_q$, contradicting the previously established concentration property.

\medskip

In the final subsection of this section, we verify that subalgebraic degree graphs have the sharp density property, completing the proof of the main theorems in this case and establishing the quantitative estimate \cref{thm:quantitative-sup}.

\medskip

\noindent\textbf{Section 4: Non-molecular graphs have the sharp density property.}
In this section we prove that the only way for an infinite subset of $\mathcal F$ to \emph{fail} to have the sharp density property is for it to contain a molecular subset, completing the proof of the main theorem.

\medskip

Our argument uses a theorem of Bourgain \cite{bourgain} formalising the heuristic that increasing events without sharp thresholds are heavily influenced by the state of a bounded number of edges. In our case the event is the existence of a giant cluster of a given density. We apply a delicate sprinkling argument to iteratively reduce the size of this bounded-size set of edges until it contains a single edge. A novel trick in this induction is that during each iteration we use the second author's universal tightness theorem \cite{uni-tightness} and the high-degree analogue of insertion-tolerance from Section 2 to stick large (but not necessarily giant) clusters to both endpoints of an edge in the current set, allowing the small number of sprinkled edges to have a disproportionately large effect. This is the most technical part of the paper.
Once the set of edges reaches a singleton, we apply Russo's formula to derive a contrasting \emph{lower bound} on the sharpness of the threshold for our event. For this lower bound to not contradict our original upper bound, the graph in question must be dense, completing the proof of \cref{thm:main_sparse}; the proof of the implication (i) $\Rightarrow$ (ii) of \cref{thm:main} in the dense case relies on a second, rather subtle application of the sprinkling technology we develop to prove that the graph must in fact be molecular. Finally we show in \cref{subsec:Bollobas} that the remaining non-trivial implication (iv) $\Rightarrow$ (i) of \cref{thm:main} follows easily from the results of \cite{dense}.

\medskip

We end the paper with some further discussion and closing remarks in \cref{section:things_to_find_a_home_for}.

\section{Lower bounds on point-to-point connection probabilities} \label{section:the_two-point_connection_property}

The goal of this section is to prove that point-to-point connection probabilities are uniformly bounded away from zero in $\eps$-supercritical percolation on finite vertex-transitive graphs, where all relevant constants depend only on the parameter $\eps>0$.  Given a finite graph $G$ and constants $0<\alpha,\delta \leq 1$, we define 
\begin{equation}
\label{eq:pcdef}
	p_c(\alpha,\delta) =p_c^G(\alpha,\delta)= \inf \bigl\{ p \in [0,1] : \p_p \bra{ \den{K_1} \geq \alpha } \geq \delta \bigr\},
\end{equation}
so that $p$ is $\eps$-supercritical if and only if $(1-\eps)p\geq p_c(\eps,\eps)$ and $\abs{V(G)} \geq 2 \eps^{-3}$.
By continuity and strict monotonicity, $p_c(\alpha,\delta)$ is equivalently the unique parameter satisfying $\p_{p_c(\alpha,\delta)} ( \den{K_1} \geq \alpha) = \delta$. 

\begin{thm}
\label{thm:2point_main}
Let $G=(V,E)$ be a finite, connected, simple, vertex-transitive graph and let $\eps>0$. If $|V| \geq 2 \eps^{-3}$ and $p\geq p_c(\eps,\eps)$ then
\[
\p_p(u \leftrightarrow v) \geq \tau(\eps):= \exp\left[-10^5\cdot\eps^{-18}\right]
\]
for every $u,v\in V$.
\end{thm}

The exact value of this bound is not important for our purposes, and we have not attempted to optimize the relevant constants. For \emph{bounded degree} graphs, a similar estimate follows from an argument essentially due to Schramm, which is recorded in \cite{benjamini-unpublished} and in more detail in \cite[Lemma 2.1]{hutchcroft+tointon}. This argument yields in particular that if $G=(V,E)$ is a finite vertex-transitive graph and $p \geq p_c(\eps,\eps)$ then 
\begin{equation}
\label{eq:bounded_degree_two_point}
\p_p(u \leftrightarrow v) \geq \exp\left[ -3 \left(\frac{2}{\eps^2}\vee \frac{1}{p}\right) \log \left(\frac{2}{\eps^2}\vee \frac{1}{p}\right)\right]
\end{equation}
for every $u,v\in V$.
This bound is adequate for our purposes in the bounded degree case, in which the condition $p \geq p_c(\eps,\eps)$ bounds $p$ away from zero when $|V|\geq 2\eps^{-3}$ by \cref{lem:pc_lower} below. As such, readers who are already familiar with \eqref{eq:bounded_degree_two_point} and are only interested in the bounded degree case of our results may safely skip the remainder of this section. The estimate \eqref{eq:bounded_degree_two_point} does \emph{not} yield a uniform lower bound on the two-point function in the high-degree case however, making a more refined analysis necessary at this level of generality. 

\begin{remark}
The assumption that $G$ is simple is not really needed for this theorem to hold: the proof works whenever $G$ has degree at most $|V|$, and yields a similar statement (with different constants) under the assumption that the degrees are bounded by $C|V|$ for some constant $C$. No such uniform two-point lower bound holds without this assumption, as can be seen by taking the product $K_n \square K_2$ and replacing each edge of $K_n$ by a large number of parallel edges.
\end{remark}

\begin{remark}
For \emph{infinite} transitive graphs, it was proven by Lyons and Schramm \cite{MR1742889} (in the unimodular case) and Tang \cite{MR4038049} (in the nonunimodular case) that there is a unique infinite cluster at $p$ if and only if $\inf_{x,y}\p_p(x \leftrightarrow y)>0$. As such, one might naively expect that we could deduce our main theorems on uniqueness directly from \cref{thm:2point_main} via a similar argument. This does not appear to be the case: firstly, we note that there \emph{do} exist large finite transitive graphs and values of $\eps$ such that there are multiple giant components with good probability at certain values of $p \geq p_c(\eps,\eps)$, despite there always being a uniform lower bound on the two-point function at such values of $p$ by \cref{thm:2point_main}. Indeed, the product $K_n \square K_2$ and the elongated torus $(\Z/n\Z)\times (\Z/2^n \Z)$ both have this property. Secondly, the proofs of \cite{MR1742889,MR4038049} both rely essentially on \emph{indistinguishability theorems} that are of an ergodic-theoretic nature and do not generalize to the finite-volume setting.
\end{remark}

\medskip

We will deduce \cref{thm:2point_main} as an analytic consequence of the following inductive lemma. Recall that we write $\|A\|=|A|/|V|$ for the density of a set of vertices in a finite graph.

\begin{lemma}[Two-point induction step] \label{lem:coarsening}
Let $G=(V,E)$ be a finite, connected, vertex-transitive graph and let $p\geq 1/(2|V|)$. If $\tau>0$ and $k\geq 1$ are such that  $\|\{v\in V: \p_p(u\leftrightarrow v) \geq \tau\}\| \geq 2^{-k}$ for every $u\in V$ then there exists $\ell \in \{0,\ldots,k-1\}$ such that
\[
\Bigl\|\Bigl\{v \in V : \p_p(u \leftrightarrow v) \geq \tau^{3^3 2^{k-\ell}} 2^{-2\ell-10} \Bigr\}\Bigr\|  \geq 2^{-\ell}
\]
for every $u\in V$.
\end{lemma}

Before proving this lemma we first state and prove some general facts that will be used in the proof. The first is a standard bound on the diameter of dense graphs whose proof we include for completeness.

\begin{lem}[Dense graphs have bounded diameter]\label{lem:dense-diam}
	Let $G=(V,E)$ be a finite, simple, connected graph. If every vertex of $G$ has degree at least $a \abs{V}$ for some $a>0$ then $\operatorname{diam} G \leq (3-a)/a$.
\end{lem}

\begin{proof}[Proof of \cref{lem:dense-diam}]
Let $a > 0$, and let $G$ be a finite, connected graph with minimum vertex degree at least $a \abs{V}$. Let $u$ and $v$ be two vertices of $G$ and let $u = u_0 , u_1, \dots, u_k = v$ be a minimal length path from $u$ to $v$. Writing $ k = 3m + r$ where $m$ is a positive integer and $r \in \{ 0 , 1 , 2 \}$, it suffices to prove that $m+1\leq 1/a$. Suppose for contradiction that this is not the case. Then
\begin{equation}
	\sum_{i=0}^m \deg u_{3i} \geq (m+1) \cdot \min_{0 \leq i \leq m} \deg u_{3i} > \abs{V},
\end{equation}
so, by the pigeonhole principle, we can find $i,j \in \{ 0,\dots , m \}$ with $i < j$ such that $u_{3i}$ and $u_{3j}$ have a common neighbour $w$. It follows that $u_0, u_1 ,\dots, u_{3i}, w , u_{3j} , \dots u_{k-1} , u_k$ is a shorter path from $u$ to $v$ than $u_0, \dots , u_k$, a contradiction. 
\end{proof}

The second ingredient we will require is a quantitative form of insertion-tolerance. In bounded degree contexts, insertion-tolerance usually refers to the fact that the conditional probability of an edge being included in the configuration given the status of every other edge is bounded away from zero. Such a statement need not be valid in regimes of interest in the high-degree case, where $p$ may be very small. Intuitively, the following proposition instead gives conditions in which we can insert exactly one edge into some sufficiently large \emph{set} of edges with good probability. This proposition will be used again in the proof of \cref{lem:molecules:attaching_clusters}.

\begin{prop}[Quantitative insertion tolerance] \label{prop:insertion-tolerance}
Let $G=(V,E)$ be a finite graph, let $p \in (0,1)$, and let $F \subseteq E$ be a collection of edges. 
Let $A \subseteq \{0,1\}^E$ be an event, let $\eta>0$ and suppose that for each configuration $\omega \in A$ there is a distinguished subset $F[\omega] $ with $F[\omega] \subseteq F \setminus \omega$ and $\abs{F[\omega]} \geq \eta \abs{F}$. If we define
	$A^+ := \{ \omega \cup \{e\} : \omega \in A \text{ and } e \in F[\omega] \}$
then
\[\p_p( A^+) \geq \frac{\eta^2}{1-p} \cdot \frac{p|F|}{p|F|+1} \cdot \p_p( A)^2.\]
\end{prop}
Note that the hypotheses of this proposition force there to be at most $(1-\eta)|F|$ open edges in $F[\omega]$ whenever the event $A$ holds.
The lower bound appearing here has not been optimized, and a more careful implementation of our argument would give a $\P_p( A)/ (-\log \P_p( A))$ term in place of the $\P_p( A)^2$ term above.
The only important conclusion of this proposition for our purposes will be that if $p|F|$, $\P_p(A)$, and $\eta$ are all bounded below by some constant $c>0$ then there exists $\delta=\delta(c)>0$ such that $\P_p^G(A^+) \geq \delta$.

\begin{proof}[Proof of \cref{prop:insertion-tolerance}]
We will abbreviate  $\p=\p_p^G$ and $\e=\e_p^G$.  Given a set of edges $H$, we write $\omega \vert_H$ for the configuration of open and closed edges in $H$, which by a standard abuse of notation we will think of both as a subset of $H$ and a function $H\to\{0,1\}$. We can sample a configuration $\omega$ with law $\p$ using the following procedure:
		\begin{enumerate}
		\item Sample the restriction $\omega|_{E\setminus F}$ of $\omega$ to $E\setminus F$.
		\item Sample a uniformly random permutation $\pi$ of the edges of $F$ independently of $\omega|_{E\setminus F}$.
		\item Sample a binomial random variable $N \sim \text{Binomial}( \abs{F} , p )$ independently of $\omega|_{E\setminus F}$ and $\pi$.
		\item Set $\omega\vert_F := \pi ( \{1,\dots N\} )$.
	\end{enumerate}
Let $\tilde \p$ denote the joint measure of $\omega$, $\pi$, and $N$ sampled as in this procedure.

\medskip

Let $A$ and $F[\omega]$ be as in the statement of the proposition. 
The assumption that $|F[\omega]| \geq \eta |F|>0$ and $F[\omega]\subseteq F \setminus \omega$ for every $\omega \in A$ guarantees that $N\leq(1-\eta)|F|<|F|$ on the event that $\omega \in A$. 
By construction we have that $\omega = \omega\vert_{E\backslash F} \cup \pi( \{1,\dots, N \} )$ so that for each $n \in \{1,\dots, \abs{F} \}$ we can rewrite
\begin{equation}
\begin{split}
	\{ \omega \in A^+\} \cap \{N=n \}&= 
	\{\omega\vert_{E \backslash F} \cup \pi(\{1, \dots, N\}) \in A^+ \} \cap \{N= n \} \\
	&= \{\omega\vert_{E \backslash F} \cup \pi(\{1, \dots, n\}) \in A^+ \} \cap \{N= n \}.
\end{split}
\end{equation}
Note that the two events on the second line are independent. One way for the union $\omega\vert_{E \backslash F} \cup \pi(\{1, \dots, n\})$ to belong to $A^+$ is for $\omega\vert_{E \backslash F} \cup \pi(\{1, \dots, n-1\})$ to belong to $A$ and for $\pi(n)$ to belong to the set $F\left[\omega\vert_{E \backslash F} \cup \pi(\{1, \dots, n-1\})\right]$. Since $\abs{F[\nu]} \geq \eta \abs{F}$ for every configuration $\nu \in A$, $\pi(n)$ belongs to this set with probability at least $\eta$ conditional on $\omega\vert_{E \backslash F}$ and $\pi(\{1, \dots, n-1\})$ and we deduce that
\begin{equation} \label{eq:insertion-tolerance-1}
\begin{split}
	\tilde \p \bra{ \omega \in A^+ \text{ and } N = n } &\geq  \eta  \tilde \p \bra{ \omega\vert_{E \backslash F} \cup \pi(\{1, \dots, n-1\}) \in A}  \tilde \p \bra{ N = n } \\
	&= \eta  \tilde \p \bra{ \omega \in A \text{ and } N = n-1  } \cdot \frac{ \tilde \p \bra{ N=n }}{ \tilde \p \bra{ N=n-1 }}.
\end{split}
\end{equation}
The ratio of probabilities appearing here is given by
\begin{equation}
\frac{ \tilde \p \bra{ N=n }}{ \tilde \p \bra{ N=n-1 }} = \frac{\binom{|F|}{n} p^n (1-p)^{|F|-n}}{\binom{|F|}{n-1} p^{n-1} (1-p)^{|F|-n+1}} = \frac{p(|F|-n+1)}{(1-p)n}
\end{equation}
and we deduce that
\begin{align}
\label{eq:insertion_key}
\tilde \p(\omega \in A^+) &\geq \eta \sum_{n=1}^{|F|} \frac{p(|F|-n+1)}{(1-p)n} \tilde \p \bra{ \omega \in A \text{ and } N = n-1  }
\nonumber\\&
= \frac{p\eta}{1-p} \tilde \e\left[ \frac{|F|-N}{N+1} \;\Big\vert\; \omega \in A\right] \tilde \p(\omega \in A),
\end{align}
where we used that $N \leq (1-\eta)|F|<|F|$ whenever $\omega\in A$ in the second line. Since $(|F|-x)/(x+1)$ is convex we may apply Jensen's inequality to deduce that
\begin{align}
\label{eq:insertion_key}
\tilde \p(\omega \in A^+) &\geq \frac{p\eta}{1-p}  \frac{|F|-\tilde \e [ N \mid \omega \in A]}{\tilde \e [ N \mid \omega \in A]+1} \tilde \p(\omega \in A).
\end{align}
Using again that $N \leq (1-\eta)|F|$ whenever $\omega \in A$ and using the bound $\tilde \e [N \mid \omega \in A] \leq \tilde \e [N]/\tilde \p(\omega \in A) = p|F|/\tilde \p(\omega \in A)$ we deduce that
\begin{align}
\label{eq:insertion_key}
\tilde \p(\omega \in A^+) &\geq \frac{\eta^2}{1-p} \cdot \frac{p|F|}{p|F|+1} \cdot \tilde \P(\omega \in A)^2,
\end{align}
concluding the proof. \qedhere

\end{proof}

We are now ready to prove \cref{lem:coarsening}.

\begin{proof} [Proof of \cref{lem:coarsening}]

Let $G^*$ be the graph with vertex set $V$ in which two vertices $u$ and $v$ are connected by an edge if and only if $\p_p(u\leftrightarrow v) \geq \tau$. The graph $G^*$ is simple, vertex-transitive, and is dense in the sense that its vertex degrees are all at least $2^{-k} |V|$. Let $C$ be a connected component of $G^*$, noting that $2^{-k} \leq \|C\| \leq 1$. Applying \cref{lem:dense-diam} to the subgraph of $G^*$ induced by $C$ implies that $\operatorname{diam}(C) \leq (3-2^{-k}/\|C\|)/(2^{-k}/\|C\|)=(3\|C\|-2^{-k})/2^{-k} \leq 3\cdot 2^k\|C\|$. 
Thus, if $u$ and $v$ belong to the same connected component of $G^*$ then there exists a sequence of vertices $u=u_0,u_1,\ldots,u_n=v$ with $n\leq 3\cdot 2^k \|C\|$ such that $u_i$ is adjacent to $u_{i-1}$ in $G^*$ for every $1\leq i \leq n$. It follows by the Harris-FKG inequality that
\begin{equation}
\P_p(u \leftrightarrow v) \geq \prod_{i=1}^n \P_p(u_{i-1} \leftrightarrow u_i) \geq \tau^{3\cdot 2^k\|C\|}=:\tau_1
\label{eq:tau1FKG}
\end{equation}
for every $u,v$ in the same connected component of $G^*$. If $G^*$ is connected then $\|C\|=1$ and the claim follows with $\ell=0$, so we may assume that $G^*$ is disconnected and that $2^{-\ell-1}\leq\|C\|<2^{-\ell}$ for some $\ell \in \{1,\ldots,k-1\}$.

\medskip

Since $G$ is connected there must exist at least one edge of $G$ connecting $C$ to $V \backslash C$. Since the connected-component equivalence relation on $G^*$ is invariant under the automorphisms of $G$, it follows by vertex-transitivity of $G$ that \emph{every} vertex in $C$ belongs to an edge from $C$ to $V \backslash C$. Letting $\partial_E C$ be the set of edges of $G$ with one endpoint in $C$ and the other in $V\setminus C$, it follows that $|\partial_E C| \geq |C|$. Since there are $|V|/|C|$ connected components of $G^*$, it follows by the pigeonhole principle that there exists a connected component $C'\neq C$ of $G^*$ such that at least $|C|^2/|V|$ edges of $|\partial_E C|$ have their other endpoint in $C'$. Let $I$ be the set of \emph{oriented} edges $e$ of $G$ with tail $e^-\in  C$ and head $e^+ \in C'$, so that $|I|\geq |C|^2/|V|$.

\medskip

Fix vertices $u\in C$ and $v \in C'$ that are the endpoints of some edge in $I$. We claim that
\begin{equation}
\p_p(u\leftrightarrow v) \geq \frac{\tau_1^8}{2^8}\|C\|^2.
\label{eq:connectingCandC'}
\end{equation}
Let $L$ be the random set of oriented edges $e\in I$ such that $u \leftrightarrow e^-$ and $v \leftrightarrow e^+$. For each $e \in I$, the Harris-FKG inequality and \cref{eq:tau1FKG} imply that $\p_p \bra{ u \leftrightarrow e^- \text{ and } v \leftrightarrow e^+} \geq \tau_1^2$ and hence that $\e_p |L| \geq \tau_1^2 |I|$. Applying Markov's inequality to $|I\setminus L|$ we obtain that
\begin{equation} \begin{split}
\p_p\left(|L| \geq \frac{\tau_1^2}{2} |I|\right) &= 1- \p_p\left(|I\setminus L| > \frac{1}{2}(2-\tau_1^2) |I|\right) \\&\geq 1- \frac{2-2\tau_1^2}{2-\tau_1^2}=\frac{\tau_1^2}{2-\tau_1^2} \geq \frac{1}{2}\tau_1^2.
\end{split}\end{equation}
Let $A$ be the event that $|L| \geq \frac{\tau_1^2}{2} |I|$ and that every edge of $L$ is closed, so that $\{u\leftrightarrow v\} \supseteq \{|L| \geq \tau_1^2|I|/2\} \setminus A$. If $\P_p(A)\leq \tau_1^2/4$ then 
\begin{equation}
\p_p(u \leftrightarrow v) \geq \p_p\left(|L| \geq \frac{\tau_1^2}{2} |I|\right) - \p_p(A) \geq \frac{1}{4}\tau_1^2,
\end{equation}
which is stronger than the claimed inequality,  so we may assume that $\P_p(A)\geq \tau_1^2/4$. In this case, applying \cref{prop:insertion-tolerance} with $F=I$, $F[\omega]=L$, and $\eta=\tau_1^2/2$ yields that
\begin{equation}
\p_p(u \leftrightarrow v) \geq \p_p(A^+) \geq \frac{\tau_1^8}{64} \cdot \frac{p|I|}{p|I|+1} \geq \frac{\tau_1^8}{64} \cdot \frac{|I|}{|I|+2|V|} \geq \frac{\tau_1^8}{64} \cdot \frac{|C|^2}{|C|^2+2|V|^2},
\end{equation}
where we used the assumption $p\geq \frac{1}{2|V|}$ in the second inequality and the inequality $|I|\geq |C|^2/|V|$ in the third. Bounding $|C|^2+2|V|^2$ by $4|V|^2$
completes the proof of \eqref{eq:connectingCandC'}.
It follows from this inequality, \eqref{eq:tau1FKG}, and a further application of Harris-FKG that 
\begin{equation}
\p_p(u \leftrightarrow w) \geq \p_p(u \leftrightarrow v) \p_p(v \leftrightarrow w) \geq \frac{\tau_1^9}{2^8} \|C\|^2
\end{equation}
for every $w \in C'$. The same inequality also holds for every $w\in C$ by \eqref{eq:tau1FKG}. Thus, recalling that $1 \leq \ell < k$ is such that $2^{-\ell-1}\leq \|C\| < 2^{-\ell}$ and using that $\tau_1=\tau^{3\cdot 2^k \|C\|}$, we deduce that
\begin{equation}
\Bigl\|\Bigl\{w \in V : \p_p(u \leftrightarrow w) \geq \tau^{3^3 2^{k-\ell}} 2^{-2\ell-10} \Bigr\}\Bigr\| \geq \|C\cup C'\| \geq 2^{-\ell},
\end{equation}
completing the proof. (The reason why we have worked so hard to get an extra factor of 2 via $\|C\cup C'\| \geq 2 \|C\|$ in the final inequality above will become clear shortly.) \qedhere
\end{proof}

\cref{lem:coarsening} implies the following general inequality by induction, from which we will deduce \cref{thm:2point_main} as a special case.

\begin{lemma} \label{lem:twopoint_general}
Let $G=(V,E)$ be a finite, connected, vertex-transitive graph and let $p\geq \frac{1}{2|V|}$. If $\tau>0$ and $k\geq 0$ are such that  $\|\{v\in V: \p_p(u\leftrightarrow v) \geq \tau\}\| \geq 2^{-k}$ for every $u\in V$ then 
\[
\p_p(u \leftrightarrow v) \geq \tau^{(54)^k} 2^{- k(54)^k}
\]
for every $u,v\in V$.
\end{lemma}

\begin{proof}[Proof of \cref{lem:twopoint_general}]
Fix $G=(V,E)$ and $p \geq \frac{1}{2|V|}$.
Applying \cref{lem:coarsening} recursively implies that there exists a decreasing sequence of non-negative integers $k=k_0 > k_1 > \cdots > k_m = 0$ with $m\leq k$ such that if we define the sequence of positive real numbers $\tau_0,\ldots,\tau_m$ recursively by $\tau_0=\tau$ and
\begin{equation}
\tau_{i+1} = \tau_i^{3^3 2^{k_i-k_{i+1}}} 2^{-2k_{i+1}-10}
\end{equation}
for each $0\leq i \leq m-1$ then 
\begin{equation}
\|\{v\in V: \p_p(u \leftrightarrow v) \geq \tau_{i} \}\| \geq 2^{-k_i}
\end{equation}
for each $1 \leq i \leq m$. It follows by induction on $i$ that
\begin{equation}
\tau_i = \tau^{3^{3i} 2^{k-k_i}} \prod_{j=1}^i \left(2^{-2k_j-10}\right)^{3^{3(i-j)} 2^{k_j-k_i}}
\end{equation}
for every $0\leq i \leq m$ and hence that
\begin{align}
\tau_m = \tau^{3^{3m}2^k} \prod_{j=1}^m \left(2^{-2k_j-10}\right)^{3^{3(m-j)} 2^{k_j}} &\geq \tau^{3^{3m}2^k} \prod_{j=1}^m \left(2^{-2k-10}\right)^{3^{3(m-j)} 2^{k}} \nonumber\\ &= \tau^{3^{3m}2^k} 2^{-(2k+10)3^{3m} 2^k \sum_{j=1}^m 3^{-3j}} \\&\geq \tau^{3^{3m}2^k} 2^{- 3^{3m} k2^k},
\end{align}
where we used the inequality $(2k+10) \sum_{j=1}^m 3^{-3j} \leq 12k \cdot (1/26) \leq k$ for $k\geq 1$ to simplify the final expression. The claim follows since $m \leq k$ and $3^3\cdot 2=54$.
\end{proof}

To deduce \cref{thm:2point_main} from \cref{lem:twopoint_general} we will need the following elementary but useful lower bound on the critical probability. 

\begin{lemma}
\label{lem:pc_lower}
If $G$ is a finite graph with maximum degree $d$ and $\eps>0$ is such that $|V| \geq 2\eps^{-3}$ then $p_c(\eps,\eps) \geq 1/2d$. In particular, if $p$ is $\eps$-supercritical then $p \geq 1/2d$.
\end{lemma}

\begin{proof}
Fix a vertex $v\in V$. For each $r$, the expected number of open simple paths of length $r$ starting at $v$ is at most $d(d-1)^{r-1} \leq d^r$ and it follows that if $p< 1/2d$ then $\E |K_v| \leq \sum_{r=0}^\infty p^r d^r <2$. On the other hand, if $p\geq p_c(\eps,\eps)$ then 
\[
\sum_{v\in V} \e_p |K_v| \geq \e_p |K_1|^2 \geq \eps^3 |V|^2,
\]
so if the inequalities $p< 1/2d$ and $p\geq p_c(\eps,\eps)$ both hold then $|V| < 2\eps^{-3}$.
\end{proof}

We are now ready to conclude the proof of \cref{thm:2point_main}.

\begin{proof}[Proof of \cref{thm:2point_main}]
Since $p\geq p_c(\eps,\eps)$ we have that $\p_p(\|K_1\| \geq \eps)\geq \eps$ and hence that $\p_p(\|K_u\|\geq \eps) \geq \eps \p_p(\|K_1\| \geq \eps)\geq \eps^2$ for every $u\in V$ by vertex-transitivity.
It follows in particular that 
 $\sum_{v\in V} \p_p(u \leftrightarrow v) = \e_p |K_u| \geq \eps^3|V|$
and hence by Markov's inequality that 
\begin{equation}
\Bigl\|\Bigl\{v\in V: \p_p(u \leftrightarrow v) \geq \frac{\eps^3}{2}\Bigr\}\Bigr\| \geq \frac{\eps^3}{2}.
\end{equation}
 Moreover, since $|V|\geq 2\eps^{-3}$ and $G$ is simple it follows from \cref{lem:pc_lower} that $p\geq 1/2d \geq 1/2|V|$. 
Thus, applying \cref{lem:twopoint_general} with $\tau=\eps^3/2$ and $k=\lceil\log_2 (2/\eps^3)\rceil$ we deduce by elementary calculations that
\begin{align}
\p_p(u\leftrightarrow v) &\geq \exp\left[-(54)^{\lceil\log_2 (2/\eps^3)\rceil} \log\frac{2}{\eps^3}- \lceil\log_2 (2/\eps^3)\rceil (54)^{\lceil\log_2 (2/\eps^3)\rceil} \right]
\nonumber\\
& \begin{multlined}[t] \geq
\exp\bigg[-54 \cdot (54)^{\log_2 (2/\eps^3)} \log\frac{2}{\eps^3}-54\cdot (54)^{\log_2 (2/\eps^3)}\log_2 \frac{2}{\eps^3} \\ -  54\cdot (54)^{\log_2 (2/\eps^3)} \bigg] \end{multlined}
\nonumber\\
&\geq \exp\left[-162  \cdot (54)^{\log_2 (2/\eps^3)} \log_2\frac{2}{\eps^3} \right]
\end{align}
for every $u,v\in V$, where we used the inequality $\log (2/\eps^3) \leq \log_2 (2/\eps^3)$ in the final inequality. We have by calculus that $x 54^x \leq 64^x / (e \log(64/54))$ for every $x\geq 0$, and using that $(64 \cdot 162) / (e \log(64/54)) = 22449.65\ldots\leq 10^5$ we deduce that
\begin{equation}
\p_p(u\leftrightarrow v) \geq \exp\left[-10^5 \cdot \eps^{-18}\right]
\end{equation}
for every $u,v\in V$ as claimed.
\end{proof}

\section{Uniqueness under the sharp density property} \label{section:deducing_uniqueness}

In this section we prove supercritical uniqueness under the assumption that our infinite set $\mathcal H \subseteq \mathcal F$ satisfies the \emph{sharp density property}, which we now introduce. This section is at the heart of the paper and contains the most significant new techniques.

\begin{defn}
	Let $G = (V,E)$ be a finite graph and let $\Delta : (0,1) \to (0,1/2]$ be decreasing. Recall from \cref{section:the_two-point_connection_property} that for each $0<\alpha,\delta<1$ we define $p_c(\alpha,\delta) =p_c^G(\alpha,\delta):= \inf \{ p \in [0,1] : \p_p (\den{K_1} \geq \alpha ) \geq \delta \}$. We say $G$ has the \emph{$\Delta$-sharp density property} if 
	\[
		\frac{ p_c ( \alpha, 1-\delta ) }{ p_c (\alpha, \delta) } \leq e^\delta \qquad \text{ for every $0<\alpha <1$ and $\Delta(\alpha)\leq \delta \leq 1/2$}.
	\]
	Let $\mathcal H \subseteq \mathcal F$ be an infinite set. Given a $\mathcal H$-indexed family $(\Delta_G)_{G \in \mathcal H}$ of decreasing (i.e.\! non-increasing) functions $\Delta_G: (0,1) \to (0,1/2]$ such that $\Delta_G \to 0$ pointwise as $G \to \infty$ in $\mathcal H$, we say that $\mathcal H$ has the \emph{$(\Delta_G)_{G \in \mathcal H}$-sharp density property} if $G$ has the $\Delta_G$-sharp density property for every $G \in \mathcal H$. We say that $\mathcal H$ has the \emph{sharp density property} if there is some $\mathcal H$-indexed family of decreasing functions $(\Delta_G)_{G \in \mathcal H}$ with $\Delta_G: (0,1) \to (0,1/2]$ and $\Delta_G\to 0$ pointwise as $G \to \infty$ in $\mathcal H$ such that $\mathcal H$ has the $(\Delta_G)_{G \in \mathcal H}$-sharp density property.  Equivalently, $\mathcal H$ has the sharp density property if and only if 
	\[
	\lim_{G \in \mathcal H} \sup_{\beta \in [\alpha,1]}	\frac{ p_c^{G}(\beta,1-\delta) }{ p_c^{G}(\beta,\delta)} = 1 
	\]
	for every $0<\alpha<1$ and $0<\delta \leq 1/2$.
\end{defn}

Graphs with subalgebraic vertex degrees can straightforwardly be shown to satisfy the sharp density property using standard sharp threshold theorems \cite{MR1371123,MR1194785,MR1303654}, all of which are proven via Fourier analysis on the hypercube. Indeed, applying these theorems in our setting leads to the following proposition, which will be used in the proof of \cref{thm:quantitative-sup} and whose proof is deferred to \cref{subsec:subpolynomial_sharpness}.

\begin{prop}
\label{prop:sharp_giant_subpolynomial}
There exists a universal constant $C$ such that the following holds.
Let $G=(V,E)$ be a finite, simple, connected vertex-transitive graph with vertex degree $d$. Then $G$ has the $\Delta$-sharp density property with 
\[
\Delta(\alpha)=\begin{cases}
 \frac{1}{2} \wedge C\sqrt{\frac{\log d}{\log |V|}} & \text{if }\alpha \geq (2/|V|)^{1/3}\\
\frac{1}{2} & \text{otherwise}.
\end{cases}
\]
In particular, if $(G_n)\seq=((V_n,E_n))\seq$ is a sequence of finite, vertex-transitive graphs with $|V_n|\to\infty$ and with subalgebraic degrees $d_n=|V_n|^{o(1)}$ then $(G_n)\seq$ has the sharp density property.
\end{prop}

A more general proposition stating that an infinite subset of $\mathcal F$ has the sharp density property if and only if it does not have a molecular subsequence is proven in \cref{section:the_sharp-giant_property}. For now we will focus on the \emph{consequences} of the sharp density property, leaving the verification of this property to later sections.

\medskip

We now state the main quantitative result of this section. When applying this theorem we will think of $\eps>0$ as a fixed constant, the number of vertices $|V|$ as being large, and $\Delta(\eps)$ as being small. We have not attempted to optimise the universal constants appearing in this theorem.

\begin{thm} \label{thm:main_single_graph}
	Let $G = (V,E)$ be a finite, simple, connected, vertex-transitive graph, let $\eps \in (0,1]$, and let $\tau(\eps)>0$ be as in \cref{thm:2point_main}. If $G$ has the $\Delta$-sharp density property for some $\Delta:(0,1)\to(0,1/2]$ then
	\[
 \p_p\left(\|K_2\|\geq \lambda \left(\frac{200\Delta(\eps)}{\eps^3 \tau(\eps)}+\frac{25}{\eps^2 \tau(\eps)|V|}\right)\right) \leq \frac{\eps}{\lambda}
	\]
for every $\eps$-supercritical parameter $p$ and every $\lambda \geq 1$.
\end{thm}

\begin{cor} \label{cor:main_sequence}
  Let $\mathcal H \subseteq \mathcal F$ be an infinite set. If $\mathcal H$ has the sharp density property, then $\mathcal H$ has the supercritical uniqueness property.
\end{cor}

In this proof, we think of an event as holding \emph{with high probability} if the probability of its complement is controlled by $\abs{V}^{-1}$ and $\Delta(x)$ for some constant $x$. Similarly, we think of a real-valued random variable as being \emph{concentrated} if the random variable lies in an interval of width controlled by $\abs{V}^{-1}$ and $\Delta(x)$ for some constant $x$ with high probability. Fix a parameter $p$ that is $\eps$-supercritical with respect to $G$. Our plan is as follows. First, in \cref{subsec:lower_parameter}, we show that it is possible to find a parameter $q$ with $p_c(\eps,\eps)\leq q \leq p$ such that  $\den{K_1}$ is concentrated in a small window under $\p_q$. Then, in \cref{subsec:concentration_implies_uniqueness}, we deduce that $\|K_2\|$ is small with high probability under $\p_q$. Finally, in \cref{subsec:sandcastles} we introduce the notion of \emph{sandcastles} and use this notion to prove that non-uniqueness of the giant cluster at the fixed parameter $p$ would contradict the established properties of percolation at the well-chosen lower parameter $q$.

\subsection{Concentration at a lower parameter}
\label{subsec:lower_parameter}

Our first step is to use the sharp density property to find another parameter $q$ with $p_c(\eps,\eps)\leq q \leq p$ such that the largest cluster under $\p_q$ is a giant whose density is concentrated in a small interval.

\begin{lem} \label{lem:inmain:concentration}
	Let $G=(V,E)$ be a finite graph with the $\Delta$-sharp density property for some $\Delta$. Then for every $\eps \in (0,1]$ and every $\eps$-supercritical parameter $p$, there exists a parameter $q \in (p_c(\eps,\eps),p)$ and a density $\alpha \geq \eps$ such that
	\eeq{eq:inmain:concentration}{
		\p_q \bra{ \abs{ \den{K_1} - \alpha } \geq \frac{4 \Delta(\eps)}{\eps} + \frac{1}{\abs{V}}} \leq 2 \Delta(\eps).
	}
\end{lem}

Roughly speaking, the idea behind the following proof is that if we pick $q$ such that the \emph{median} - or any other particular quantile - of the density of the largest cluster increases slowly across a small neighbourhood of $q$ then the sharp density property implies that the density of the giant at $q$ must be concentrated; such a $q$ can always be found since a bounded increasing function cannot increase rapidly everywhere.

\begin{proof}[Proof of \cref{lem:inmain:concentration}]
We may assume $4 \Delta(\eps) \leq \eps$, the lemma being trivial otherwise.
	Consider the increasing sequence of reals $q_0,q_1,\dots$ given by
		$q_j := e^{j\Delta(\eps)} p_c(\eps,\eps)$ for each $j\geq 0$, and let $k$ be the maximum integer such that $q_{2k} \leq p$. We start by finding a simple lower bound for $k$. Since $p$ is $\eps$-supercritical, we know $(1-\eps)^{-1} \cdot p_c(\eps,\eps) \leq p$ and hence that $k \geq r$ for any integer $r$ satisfying $e^{2r\Delta(\eps)} \leq (1-\eps)^{-1}$. It follows in particular that
\begin{equation}
\label{eq:finding_q_k_lower}
	k \geq \left \lfloor \frac{1}{2} \cdot \frac{\log 1/(1-\eps)}{\Delta(\eps)} \right \rfloor 
	\geq \left \lfloor \frac{1}{2} \cdot \frac{\log (1+\eps)}{\Delta(\eps)} \right \rfloor 
	\geq \left \lfloor \frac{\eps}{4 \Delta(\eps)} \right \rfloor \geq \frac{\eps}{8\Delta(\eps)},
 \end{equation}
where we used the inequality $1/(1-x)\geq 1+x$ for $0\leq x <1$ in the first inequality, the inequality $\log (1+x) \geq x/2$ for $0 \leq x \leq 1$ in the second inequality, and the assumption $4 \Delta(\eps) \leq \eps$ in the final inequality.

Now, for each $i\geq 0$ we define the density $\lambda_i$ of $K_1$ under $\P_{q_i}$ by
\[
	\lambda_i := \max \{ \beta \in [0,1] : \p_{q_i} \bra{ \den{K_1} \geq \beta } \geq \eps\}, 
\]
so that $\lambda_i \geq \lambda_0  =\eps$ for every $i\geq 0$.
Since $\lambda_i$ is increasing in $i$ we have that
\[
	\sum_{i=1}^{k} |\lambda_{2i} - \lambda_{2(i-1)}| = \sum_{i=1}^{k} \lambda_{2i} - \lambda_{2(i-1)} = \lambda_{2k} - \lambda_0 \leq 1,
\]
and hence by the pigeonhole principle that there is some $j \in \{ 1 , \dots , k \}$ such that
\[
	|\lambda_{2j} - \lambda_{2(j-1)}| = \lambda_{2j} - \lambda_{2(j-1)} \leq \frac{1}{k} \leq \frac{8\Delta(\eps)}{\eps},
\]
where the final inequality follows from \eqref{eq:finding_q_k_lower}.

We will argue that the values
\[
q = q_{2j-1} \qquad \text{ and } \qquad \alpha = \frac{\lambda_{2j} + \lambda_{2(j-1)}}{2}
\]
satisfy the conclusions of the lemma.
Indeed, by definition of $\lambda$, we have that
\begin{equation} \label{eq:inmain:concentration:above_and_below}
	\p_{q_{2(j-1)}} \bra{ \den{K_1} \geq \lambda_{2(j-1)} } \geq \eps \qquad \text{but} \qquad \p_{q_{2j}} \bra{ \den{K_1} \geq \lambda_{2j} + \abs{V}^{-1} } < \eps.
\end{equation}
We also have by assumption that $4\Delta(\eps) \leq \eps$ and $\eps \leq 1/2$, so that $\Delta(\eps) \leq \eps \leq 1-\Delta(\eps)$ and hence by the definiton of the $\Delta$-sharp density property that
\[
	\max\left\{\frac{p_c \bra{ \beta , 1- \Delta(\eps) }}{p_c \bra{ \beta ,\eps }},\frac{p_c \bra{ \beta , \eps }}{p_c \bra{ \beta ,\Delta(\eps) }}\right\} \leq \frac{p_c \bra{ \beta , 1- \Delta(\eps)) }}{p_c \bra{ \beta ,\Delta(\eps) }}  \leq e^{\Delta(\eps)}
\]
 for every $\beta \in [\eps,1]$. Applying this inequality with the values $\beta=\lambda_{2(j-1)}$ and $\beta=(\lambda_{2j}+\abs{V}^{-1})\wedge 1$ and using that $e^{-\Delta(\eps)} q_{2j} = q = e^{\Delta(\eps)} q_{2(j-1)}$, we deduce from \eqref{eq:inmain:concentration:above_and_below} that 
\[
	\p_{q} \bra{ \den{K_1} \geq \lambda_{2(j-1)} } \geq 1-\Delta(\eps) \qquad \text{and} \qquad \p_q \bra{ \den{K_1} \geq \lambda_{2j} + \abs{V}^{-1}} \leq \Delta(\eps).
\]
Since $|\alpha-\lambda_{2(j-1)}|$ and $|\alpha-\lambda_{2j}|$ are both bounded by $4 \Delta(\eps)/\eps$, it follows that
\[
	\p_q \bra{ \abs{ \den{K_1} - \alpha } \geq \frac{4 \Delta(\eps)}{\eps} + \frac{1}{\abs{V}} } \leq 2 \Delta(\eps)
\]
as claimed.
\end{proof}

\subsection{Concentration implies uniqueness}
\label{subsec:concentration_implies_uniqueness}

By applying \Cref{lem:inmain:concentration} with our fixed parameter $p$, we obtain a parameter $q \in (p_c(\eps,\eps),1)$ and a density $\alpha \geq \eps$ that satisfy \eqref{eq:inmain:concentration}. We next argue that concentration of $\den{K_1}$ under $\p_q$ implies uniqueness of the giant cluster under $\p_q$.

\begin{lem} \label{lem:inmain:concentration_implies_uniqueness}
	Let $G=(V,E)$ be a finite graph, let $q \in (0,1]$, and let $\tau := \min_{u,v \in V} \p_q \bra{ u \leftrightarrow v }$. The estimate
	\[
		\p_q \bra{ \den{K_2} \geq 2 \delta } \leq \bra{ 1+ \frac{1}{4 \delta^2 \tau} } \p_q \bra{ \abs{ \den{K_1} - \alpha } \geq \delta}
	\]
	holds for every $\alpha,\delta >0$.
\end{lem}

The idea is that by the two-point connection property (and positive association), on any increasing event, we can connect $K_1$ and $K_2$ to form a new largest cluster $K_1$ with good probability. Thus, given that $\den{K_1}$ is concentrated, $\den{K_1 \sqcup K_2}$ must be close to $\den{K_1}$ with high probability. Equivalently, $\den{K_2} = \den{K_1 \sqcup K_2} - \den{K_1}$ must be close to zero with high probability.

\begin{proof}[Proof of \cref{lem:inmain:concentration_implies_uniqueness}]
By the union bound,
\[
	\p_q \bra{ \den{K_2} \geq 2 \delta \text{ and } \den{K_1} \geq \alpha - \delta} \geq \p_q \bra{ \den{K_2} \geq 2 \delta } - \p_q \bra{ \abs{ \den{K_1} - \alpha } \geq \delta }. 
\]
On the event that $\den{K_2} \geq 2 \delta$ and $\den{K_1} \geq \alpha - \delta$, every pair of vertices $u,v$ with $u \in K_2$ and $v \in K_1$ has $\den{ K_u \cup K_v } \geq \alpha + \delta$, and there are at least $4\delta^2 \abs{V}^2$ such pairs. So, by linearity of expectation,
\begin{align}
	\max_{u,v\in V} \p_q \bra{ \den{K_u \cup K_v} \geq \alpha + \delta } &\geq  \frac{1}{\abs{V}^2} \sum_{u,v \in V} \p_q \bra{ \den{K_u \cup K_v} \geq \alpha + \delta } \nonumber\\ &\geq 4 \delta^2 \left[\p_q \bra{ \den{K_2} \geq 2 \delta } - \p_q \bra{ \abs{ \den{K_1} - \alpha } \geq \delta }\right].
\end{align}
 When $\den{K_u \cup K_v} \geq \alpha + \delta$ and $u \leftrightarrow v$, we are guaranteed to have $\den{K_1} \geq \alpha + \delta$ and hence that $\abs{ \den{K_1} - \alpha } \geq \delta$. It follows by Harris's inequality that
\begin{align}
	\p_q \bra{ \abs{ \den{K_1} - \alpha } \geq \delta } &\geq \max_{u,v\in V}\p_q \bra{\den{K_u \cup K_v} \geq \alpha + \delta} \cdot \p_q \bra{ u \leftrightarrow v } \nonumber\\
	&\geq 4 \delta^2 \tau \left[\p_q \bra{ \den{K_2} \geq 2 \delta } - \p_q \bra{ \abs{ \den{K_1} - \alpha } \geq \delta }\right],
\end{align}
and the claim follows by rearranging.
\end{proof}

\subsection{Proof of \cref{thm:main_single_graph} via sandcastles}
\label{subsec:sandcastles}

So far we have obtained good control over $\den{K_1}$ and $\den{K_2}$ under $\p_q$, where $q$ is a well-chosen parameter $p_c(\eps,\eps) \leq q \leq p$. We now need to convert this into an upper bound on the probability that the second largest cluster is large under $\p_p$. We do this by introducing an object we call a \emph{sandcastle}. This is defined in terms of the canonical monotone coupling $(\omega_q,\omega_p)$ of the percolation measures $\p_q$ and $\p_p$ with $q \leq p$ on any given graph, where each closed edge of $\omega_p$ is also closed in $\omega_q$ and each open edge of $\omega_p$ is open in $\omega_q$ with probability $q/p$. We write $\p_{q,p}$ for the joint law of this coupling. (Recall that in this coupling, when we condition on $\omega_p$, the states of the edges in $\omega_q$ are still independent of each other.) Informally, a sandcastle is a large connected subgraph of $G$ with the property that even knowing that the subgraph is entirely open in $\omega_p$, there remains a good condtional probability that it contains no large cluster for $\omega_q$. We fix this `good probability' to be $1/2$ in the following definition, but we could have used any other universal constant in $(0,1)$.

\begin{defn}
	Let $G=(V,E)$ be a finite graph. Let $0\leq q \leq p\leq 1$ and let $0\leq \alpha,\beta \leq 1$. A $[(p,\beta) \to (q,\alpha)]$\emph{-sandcastle} is a connected subgraph $S \subseteq G$ such that 
		$\den{S} \geq \beta$ and 
		\[\p_{q,p} \bra{ \den{ K_1(\omega_q \cap S) } < \alpha \mid S \subseteq \omega_p } \geq \frac{1}{2}.\]
\end{defn}

We now show that non-uniqueness of the giant cluster under $\p_p$ and uniqueness of the giant cluster under $\p_q$ together imply that some cluster must be a sandcastle with good probability under the measure $\p_p$.

\begin{lem} \label{lem:inmain:sandcastles:existence}
	Let $G=(V,E)$ be a finite graph. For each $0 \leq q \leq p\leq 1$, and $0<\alpha,\beta<1$ there exists a vertex $u$ such that
	\begin{multline*}
		\p_p \bra{ K_u \text{ is a } [(p,\beta) \to (q,\alpha)]\text{-sandcastle} } \\\geq \beta \sqbra{ \p_p \bra{ \den{K_2} \geq \beta } - 4 \p_q \bra{ \den{K_2} \geq \alpha } }.
	\end{multline*}
\end{lem}

\begin{proof}[Proof of \cref{lem:inmain:sandcastles:existence}]
	Consider a configuration $\nu\in\{ 0,1\}^E$ in which $\den{K_2} \geq \beta$ but no clusters are $[(p,\beta) \to (q,\alpha)]$-sandcastles. Let $A := K_1(\nu)$ and $B := K_2(\nu)$. Since $\den{A}\geq \beta$ but $A$ is not a $[(p,\beta) \to (q,\alpha)]$-sandcastle, we know by the definition of sandcastles that
	\[
		\p_{q,p} \bra{ \den{K_1( \omega_q \cap A )} \geq \alpha \mid \omega_p = \nu }=\p_{q,p} \bra{ \den{K_1( \omega_q \cap A )} \geq \alpha \mid A \subseteq \omega_p } \geq \frac{1}{2}.
	\]
	The same result holds for $B$. Since $A$ and $B$ are disjoint, the restrictions of $\omega_q$ to $A$ and $B$ are conditionally independent given $\omega_p$, hence
	\[
		\p_{q,p} \bra{ \den{K_1( \omega_q \cap A )} \geq \alpha \text{ and } \den{K_1( \omega_q \cap B )} \geq \alpha \mid \omega_p = \nu } \geq \frac{1}{4}.
	\]
	The edges in the boundary of $A$ are all closed in $\nu$, disconnecting $A$ from $B$. Since $\omega_q \leq \omega_p$,  these edges are also closed in $\omega_q$ when $\omega_p=\nu$. In particular, given that $\omega_p = \nu$, the subgraphs $K_1( \omega_q \cap A )$ and $K_1( \omega_q \cap B )$ are not connected to each other in $\omega_q$, hence
	\begin{multline*}
		\p_{q,p} \bra{ \den{K_2( \omega_q )} \geq \alpha \mid \omega_p = \nu } \\\geq \p_{q,p} \bra{ \den{K_1( \omega_q \cap A )} \geq \alpha \text{ and } \den{K_1( \omega_q \cap B )} \geq \alpha \mid \omega_p = \nu } \geq \frac{1}{4}.
	\end{multline*}
	Letting $\mathcal E$ be the event that $\|K_2(\omega_p)\| \geq \beta$ but no cluster in $\omega_p$ is a $[(p,\beta) \to (q,\alpha)]$-sandcastle, it follows since $\nu \in \cE$ was arbitrary that
	\[
		\p_{q,p} \bra{ \den{K_2( \omega_q )} \geq \alpha \mid \omega_p \in \mathcal E} \geq \frac{1}{4}.
	\]
	It follows from this and a union bound that
	\spliteq{
		\p_q \bra{ \den{K_2} \geq \alpha } &\geq \p_{q,p} \bra{ \den{K_2( \omega_q )} \geq \alpha \mid \omega_p \in \mathcal E} \cdot \p_p \bra{ \mathcal E }\\
		&\geq \begin{multlined}[t]\frac{1}{4} \bigg( \p_p \bra{ \den{K_2} \geq \beta } \\- \p_p \bra{ \text{some cluster is a } [(p,\beta) \to (q,\alpha)]\text{-sandcastle}} \bigg),\end{multlined}
	}
	which rearranges to give that
	\begin{multline}
	\label{eq:sandcastle_union_bound}
		\p_p \bra{ \text{some cluster is a } [(p,\beta) \to (q,\alpha)]\text{-sandcastle}} \\\geq \p_p \bra{ \den{K_2} \geq \beta } - 4 \p_q \bra{ \den{K_2} \geq \alpha }.
	\end{multline}
	Every cluster that is a $[(p,\beta) \to (q,\alpha)]$-sandcastle contains at least $\beta \abs{V}$ vertices by definition, and we deduce by linearity of expectation that
	\begin{align*}
		\max_{u\in V} \p_p &\bra{ K_u \text{ is a } [(p,\beta) \to (q,\alpha)] \text{-sandcastle} } \\&\geq 
		\frac{1}{|V|}\sum_{u\in V} \p_p \bra{ K_u \text{ is a } [(p,\beta) \to (q,\alpha)] \text{-sandcastle} }
		\nonumber\\ &\geq
		\beta \p_p \bra{ \text{some cluster is a } [(p,\beta) \to (q,\alpha)]\text{-sandcastle}}.
	\end{align*}
	The claimed inequality follows from this and \eqref{eq:sandcastle_union_bound}.
\end{proof}

We want to use the fact that $K_u$ is a sandcastle with good probability to contradict the concentration of $\den{K_1}$ under $\p_q$. The rough idea is as follows: as we pass from $\omega_p$ to $\omega_q$, with good probability this sandcastle disintegrates into only small clusters, none of which are equal to the giant cluster $K_1(\omega_q)$. Since the status of any edge that does not touch the cluster of the vertex $u$ in $\omega_p$ remains conditionally distributed as Bernoulli percolation, this implies that there exists a large set of vertices whose complement contains, with good probability, an $\omega_q$ cluster whose density is close to the typical density of the largest cluster in the whole graph. Using Harris' inequality and uniqueness of the giant cluster in $\omega_q$, we deduce that $\den{K_1(\omega_q)}$ is abnormally high with good probability, contradicting the concentration of the giant cluster's density under $\p_q$.

\medskip

We now begin to make this argument precise.
 For each subgraph $H$ of $G$, let $\overline{H}$ denote the set of all edges that have at least one endpoint in the vertex set of $H$.

\begin{lemma}
\label{lem:localization}
Let $G=(V,E)$ be a finite, vertex-transitive graph and let $q\in (0,1]$. The estimate
\begin{multline}
  \P_q\Bigl(\bigl\|K_1\bigl(\omega \setminus \overline{H}\bigr)\bigr\| \geq \beta \Bigr) \cdot \frac{2\beta \P_q(\|K_1\|\geq \beta)-\beta^2}{2-\beta^2}  \\\leq   \p_q \bra{ \|K_1\| \notin \Big(\beta,\beta + \frac{\beta^2}{2}\|H\|\Big) } + \p_q \bra{ \den{K_2} \geq \beta}
\end{multline}
holds for every subgraph $H$ of $G$ and every $0<\beta<1$.
\end{lemma}

\begin{proof}[Proof of \cref{lem:localization}]
Let $X$ be the set of vertices that are contained in clusters with density at least $\beta$, noting that
\begin{equation}
	\p_q \bra{ X \not= K_1 } \leq \p_q \bra{ \den{K_1} \leq \beta} + \p_q \bra{ \den{K_2} \geq \beta}
\end{equation}
and hence that
\begin{multline}
\label{eq:X_to_K1K2}
\p_q \bra{ \den{X} \geq \beta + \frac{\beta^2}{2}\|H\| } \leq \p_q \bra{ \|K_1\| \geq \beta + \frac{\beta^2}{2}\|H\| } \\+ \p_q \bra{ \den{K_1} \leq \beta} + \p_q \bra{ \den{K_2} \geq \beta}.
\end{multline}

We have by Markov's inequality applied to $\|H \setminus X\|$ that
\begin{align}
	\p_q \bra{ \den{X \cap H} \geq \frac{\beta^2}{2}\|H\|  }&=1-\p_q \bra{ \den{H \setminus X} > \left(1- \frac{\beta^2 }{2}\right)\|H\|}
	\nonumber\\ &\geq 1-\left(1-\frac{\beta^2 }{2} \right)^{-1} \|H\|^{-1} \e_q \den{H\setminus X}\nonumber\\
	&= 1-\left(1-\frac{\beta^2 }{2} \right)^{-1} \P_q(\|K_u\|<\beta),
\end{align}
where $u$ is an arbitrary vertex and we used vertex-transitivity in the last line. Bounding $\P_q(\|K_u\| \geq \beta) \geq \beta \P_q(\|K_1\|\geq \beta)$ we deduce that
\begin{align}
	 \p_q \bra{ \den{X \cap H} \geq \frac{\beta^2}{2}\|H\|  }&\geq  1-\frac{2-2\beta\P_q(\|K_1\|\geq \beta)}{2-\beta^2} \\&= \frac{2\beta \P_q(\|K_1\|\geq \beta)-\beta^2}{2-\beta^2}
\end{align}
and hence by Harris' inequality that
\begin{align}
	\p_q \bra{ \den{X} \geq \beta + \frac{\beta^2}{2}\|H\| } &\geq \p_q \bra{ \den{X \setminus H} \geq \beta } \cdot\p_q \bra{ \den{X \cap H} \geq \frac{\beta^2}{2}\|H\|  } \nonumber\\
	& \geq  \p_q \bra{ \|K_1(\omega \setminus \overline{H})\| \geq \beta } \cdot \frac{2\beta \P_q(\|K_1\|\geq \beta)-\beta^2}{2-\beta^2}.\label{eq:X_uppertail}
\end{align}
The claim follows by combining \eqref{eq:X_to_K1K2} and \eqref{eq:X_uppertail}.
\end{proof}

\begin{proof}[Proof of \cref{thm:main_single_graph}]
Write $\tau=\tau(\eps)$ and $\Delta=\Delta(\eps)$, fix an $\eps$-supercritical parameter $p$, and let $q$ and $\alpha$ satisfying $\alpha \geq \eps$ be as in \cref{lem:inmain:concentration}. 
Define
\[
\delta=\frac{4 \Delta}{\eps} + \frac{1}{|V|} \qquad \text{ and } \qquad \beta_0 = \frac{25\delta}{\eps^2 \tau}= \frac{200\Delta}{\eps^3 \tau}+\frac{25}{\eps^2 \tau |V|},
\]
and fix some $\beta \geq \beta_0$. This value of $\delta$ is chosen so that
\begin{equation}
\label{eq:sandcastle_delta_def_reason}
\p_q\Bigl(\bigl|\|K_1\|-\alpha\bigr| \geq \delta\Bigr) \leq 2\Delta
\end{equation}
by \cref{lem:inmain:concentration}.
We will refer to $[(p,\beta) \to (q,\eps/2)]$-sandcastles simply as sandcastles for the remainder of the proof. 
It suffices to prove that
\[\p_p(\|K_2\|\geq \beta) < \frac{200 \Delta}{\eps^2\tau\beta} \leq  \eps\cdot \frac{\beta_0}{\beta},\] so we will suppose for contradiction that the reverse inequality 
\begin{equation}
\label{eq:sandcastle_contradiction_assumption}
\p_p(\|K_2\|\geq \beta) \geq \frac{200 \Delta}{\eps^2\tau \beta}
\end{equation}
holds. Since $\|K_2\|\leq 1$, in this case we must have that $\beta_0 \leq \beta\leq 1$ and hence that $\delta \leq \eps^2/25 \leq \eps/2$ and $\Delta \leq 1/200$.

\medskip

 Since $q \geq p_c(\eps,\eps)$, we can apply \cref{thm:2point_main} to bound the minimal connection probability $\min_{u,v} \p_q \bra{ u\leftrightarrow v} \geq \tau=\tau(\eps)$. Thus, applying \Cref{lem:inmain:concentration_implies_uniqueness} yields that
\begin{equation}
\label{eq:inmain:concentration_implies_uniqueness:final} \begin{split}
	\p_q \bra{ \den{K_2} \geq \frac{\eps}{2} } &\leq \bra{1+\frac{4}{\eps^2 \tau}}\p_q \bra{ \abs{\den{K_1} - \alpha} \geq \frac{\eps}{2}} \\&\leq \bra{1+\frac{4}{\eps^2 \tau}} \cdot 2 \Delta \leq \frac{10 \Delta}{ \eps^2 \tau },
\end{split}\end{equation}
where we used the assumption $\eps/2 \geq \delta$ and \cref{eq:sandcastle_delta_def_reason} in the second inequality.
Applying  \Cref{lem:inmain:sandcastles:existence}, we deduce that there exists a vertex $u$ such that
\eeq{eq:inmain:sandcastles:final}{
	\p_p \bra{ K_u \text{ is a sandcastle} } \geq \beta \p_p \bra{ \den{K_2} \geq \beta } - \frac{40 \beta \Delta }{\eps^2 \tau} \geq \frac{\beta}{2} \p_p \bra{ \den{K_2} \geq \beta },
}
where we used the assumption \eqref{eq:sandcastle_contradiction_assumption} in the final inequality.
By vertex-transitivity, this holds for \emph{every} vertex $u\in V$. 
Fix a vertex $u\in V$ and let $\mathscr{S}_u$ be the event that $K_u(\omega_p)$ is a sandcastle. Since $\omega_q \leq \omega_p$, no vertex in $K_u(\omega_p)$ is connected to a vertex of $V\setminus K_u(\omega_p)$ in $\omega_q$, and using the fact that $\alpha - \delta \geq \eps/2$ (because $ \alpha \geq \eps$ and $\delta \leq \eps /2$), we have the inclusion of events
\begin{multline*}
\Bigl\{\den{K_1( \omega_q \cap K_u(\omega_p) )} < \eps/2\Bigr\} \cap \Bigl\{\den{K_1(\omega_q)} \geq \alpha - \delta\Bigr\} \\ \subseteq \Bigl\{\Bigl\|K_1\Bigl( \omega_q \setminus \overline{ K_u(\omega_p) } \Bigr) \Bigr\| \geq \alpha - \delta\Bigr\}.
\end{multline*}
Taking probabilities and using the definition of sandcastles, we deduce that
\begin{align}
	&\p_{q,p} \bra{ \den{ K_1( \omega_q \setminus \overline{ K_u(\omega_p) } )}  \geq \alpha - \delta \bigm| \mathscr{S}_u} \nonumber\\
	&\hspace{0.5cm} \geq \p_{q,p} \bra{ \den{ K_1(\omega_q \cap K_u(\omega_p)) } < \frac{\eps}{2} \bigm| \mathscr{S}_u} - \p_{q,p} \bra{ \den{K_1(\omega_q)} \leq \alpha - \delta \mid \mathscr{S}_u} \nonumber\\
	& \hspace{1cm} \geq \frac{1}{2} - \p_{q,p} \bra{ \den{K_1(\omega_q)} \leq \alpha - \delta \mid \mathscr{S}_u}.
\end{align}
Using \eqref{eq:inmain:concentration} and \eqref{eq:inmain:sandcastles:final}, we can bound the error term
\begin{align}
	\p_{q,p} \bra{ \den{K_1(\omega_q)} \leq \alpha - \delta \mid \mathscr{S}_u} &\leq \frac{ \p_q \bra{ \abs{ \den{K_1}-\alpha} \geq \delta } } { \p_p \bra{ K_u \text{ is a sandcastle}} } \\
	&\leq \frac{4 \Delta}{\beta \p_p(\|K_2\|\geq \beta)} \leq \frac{1}{4}
\end{align}
by the assumption that $\p_p \bra{ \den{K_2} \geq \beta } \geq 200 \Delta/(\eps^2 \tau \beta) \geq 16\Delta / \beta$,
so that
\begin{align}
\label{eq:sandcastle1}
	\p_{q,p} \bra{ \den{ K_1( \omega_q \backslash \overline{ K_u(\omega_p) } )}\geq \alpha - \delta \mid \mathscr{S}_u}  
	 \geq \frac{1}{4}.
\end{align}
Since the left hand side of \eqref{eq:sandcastle1} can be written as a weighted sum of conditional probabilities given that $K_u(\omega_p)$ is equal to a \emph{specific} sandcastle, there must exist a sandcastle $S$ such that
\begin{multline}
\label{eq:sandcastle2}
	\p_{q,p} \bra{ \den{ K_1\bigl( \omega_q \backslash \overline{ S } \bigr)}\geq \alpha - \delta \mid K_u(\omega_p)=S} \\ =
	\p_{q,p} \bra{ \den{ K_1\bigl( \omega_q \backslash \overline{ K_u(\omega_p) } \bigr)}\geq \alpha - \delta \mid K_u(\omega_p)=S}  \geq \frac{1}{4}.
\end{multline}
Since the event $\{K_u(\omega_p)=S\}$ depends only on the status of edges in $\overline {S}$, it is independent of the restriction of $\omega_q$ to $E \backslash \overline{S}$, and we deduce that
\begin{align}
	\p_{q} \bra{ \den{ K_1( \omega \backslash \overline{ S } )}  \geq \alpha - \delta} &= \p_{q,p} \bra{ \den{ K_1\bigl( \omega_q \backslash \overline{ S } \bigr)}\geq \alpha - \delta \mid K_u(\omega_p)=S} \geq \frac{1}{4}.
	\label{eq:sandcastle_contradiction_1}
\end{align}
On the other hand, using that $\Delta \leq 1/200$  and hence that \[2 \P_q(\|K_1\|\geq (\alpha-\delta)) \geq 2(1-2\Delta) \geq \alpha - \delta,\]  \cref{lem:localization} implies that
\begin{multline}
\p_{q} \bra{ \den{ K_1( \omega \backslash \overline{ S } )}  \geq \alpha - \delta} \leq \frac{2-(\alpha-\delta)^2}{2(\alpha-\delta) \P_q(\|K_1\|\geq \alpha-\delta)-(\alpha-\delta)^2}\\ \cdot \left[\p_q \bra{ \|K_1\| \notin \Big(\alpha-\delta,\alpha-\delta + \frac{(\alpha-\delta)^2}{2}\beta\Big) } + \p_q \bra{ \den{K_2} \geq \alpha-\delta}\right],
\end{multline}
and since $\alpha - \delta \geq \eps/2 \leq 1/2$ and $\beta \geq 16\eps^{-2} \delta$, and \[\P_q(\|K_1\|\geq \alpha-\delta)\geq 1-2\Delta \geq 99/100,\] it follows that
\begin{multline}
\p_{q} \bra{ \den{ K_1( \omega \backslash \overline{ S } )}  \geq \alpha - \delta} \leq \frac{8}{4\eps \P_q(\|K_1\|\geq \alpha-\delta)-\eps^2}\\ \cdot \left[\p_q \bra{ \bigl|\|K_1\| -\alpha\bigr| \geq \delta } + \p_q \bra{ \den{K_2} \geq \eps/2}\right].
\end{multline}
Applying \eqref{eq:sandcastle_delta_def_reason} and \eqref{eq:inmain:concentration_implies_uniqueness:final} to control the two probabilities appearing here we obtain that
\begin{equation}
\label{eq:sandcastle_contradiction_2}
\p_{q} \bra{ \den{ K_1( \omega \backslash \overline{ S } )}  \geq \alpha - \delta} \leq \frac{8}{4\eps (1-2\Delta-\eps/4)} \cdot \left[2\Delta + \frac{10\Delta}{\eps^2\tau}\right] \leq \frac{48\Delta}{\eps^3 \tau}
\end{equation}
where we used that $\Delta \leq 1/200 \leq 1/8$ in the final inequality. The two estimates \eqref{eq:sandcastle_contradiction_1} and \eqref{eq:sandcastle_contradiction_2} contradict each other since $200\Delta/\eps^3 \tau \leq \beta_0 \leq 1$. \qedhere

\end{proof}

\subsection{Subalgebraic degree graphs have the sharp density property}
\label{subsec:subpolynomial_sharpness}

In this section we prove \cref{prop:sharp_giant_subpolynomial}. As stated above, this proposition is a straightforward application of standard sharp-threshold theorems. The details of the implementation of the proof are somewhat technical but do not contain any significant new ideas.
We will apply the following straightforward consequence of the results of Talagrand \cite{MR1303654}. We refer the reader to \cite[Chapter 4]{MR3751350} and \cite{MR3443800} for general background on sharp threshold theorems.

\begin{thm}
\label{thm:KKL_subpolynomial}
Let $E$ be a finite set and let $A\subseteq \{0,1\}^E$ be an increasing event. Let $\Gamma$ be a group acting on $E$ and for each $e\in E$ let $\Gamma e=\{\gamma e : \gamma \in \Gamma\}$ be the orbit of $e$ under $\Gamma$. There exists a universal constant $c>0$ such that if $A$ is invariant under the action of $\Gamma$ on $\{0,1\}^E$ then
\[
\frac{\dif}{\dif p}\p_p(A) \geq c \left[p(1-p) \log \frac{2}{p(1-p)}\right]^{-1}\p_p(A)(1-\p_p(A)) \log \left(2\min_{e\in E}|\Gamma e|\right)
\]
for every $p\in (0,1)$.
\end{thm}

\begin{proof}[Proof of \cref{thm:KKL_subpolynomial}]
The \textbf{influence} of an edge $e$ with respect to $A$ under $\p_p$ is defined to be
\[
I_p(A,e):=\p_p( \omega \cup \{e\} \in A,\; \omega \setminus \{e\} \notin A).
\]
Russo's formula states that if $A$ is an increasing event then
\begin{equation}\label{eq:Russo}
\frac{\dif}{\dif p}\p_p(A) = \sum_{e\in E} I_p(A,e)
\end{equation}
for every $p\in [0,1]$. It is a theorem of Talagrand \cite{MR1303654} that there exists a universal constant $0<c\leq 1$ such that if $A$ is increasing then
\begin{equation}
\label{eq:Talagrand}
p(1-p) \log \left(\frac{2}{p(1-p)}\right)\sum_{e\in E} \frac{I_p(A,e)}{\log \frac{1}{p(1-p)I_p(A,e)}} \geq c\cdot \p_p(A)(1-\p_p(A))
\end{equation}
and hence that
\begin{multline}
\label{eq:Talagrand2}
\sum_{e\in E} I_p(A,e) \geq c\cdot \p_p(A)(1-\p_p(A))\\ \cdot \left[p(1-p) \log \frac{2}{p(1-p)}\right]^{-1}  \log \frac{1}{p(1-p) \max_e I_p(A,e)}.
\end{multline}
(Note that Talagrand states his inequality in terms of \emph{open} pivotals, so that his expression differs from ours by some factors of $1/p$.)
Intuitively, this inequality implies that any event that does not depend too strongly on the status of any particular edge must have a sharp threshold, i.e., must have probability changing rapidly from near $0$ to near $1$ over a short interval.
 Letting $e$ maximize the influence, we have by \eqref{eq:Russo} and \eqref{eq:Talagrand2} that 
\begin{equation} \begin{split}
\frac{\dif}{ \dif p}\p_p(A) 
&\geq \frac{1}{p(1-p)} \cdot\max\bigg\{|\Gamma e| p(1-p)I_p(A,e),\\ &\hspace{1cm} \, c\cdot \p_p(A)(1-\p_p(A))\cdot \left[ \log \frac{2}{p(1-p)}\right]^{-1} \log \frac{1}{p(1-p)I_p(A,e)} \bigg\}.
\end{split}\end{equation}
Since the function $f(x)=\max \{ax,b\log 1/x\}$ attains its minimum when $\frac{1}{x}\log\frac{1}{x}=\frac{a}{b}$, it follows that
\begin{multline}
\frac{\dif}{\dif p}\p_p(A) \geq \frac{c}{p(1-p)}\cdot \p_p(A)(1-\p_p(A)) \\\cdot \left[ \log \frac{2}{p(1-p)}\right]^{-1} W\!\left(\frac{|\Gamma e| \log \frac{2}{p(1-p)}}{c\cdot \p_p(A)(1-\p_p(A))} \right)
\end{multline}
where $W$ is the Lambert W-function (i.e., the inverse function of $x e^x$). The claim follows since $\p_p(A)(1-\p_p(A))\leq 1$ and $W$ is increasing and satisfies $W(x) \geq \frac{1}{2} \log x$ for every $x\geq 1$.
\end{proof}

We now apply \cref{thm:KKL_subpolynomial} to prove \cref{prop:sharp_giant_subpolynomial}.

\begin{proof}[Proof of \cref{prop:sharp_giant_subpolynomial}]
Let $G=(V,E)$ be a finite vertex-transitive graph of degree $d \geq 2$. 
It follows from \cref{lem:pc_lower} that $p_c(\alpha,\delta) \geq 1/2d$ for every $\alpha,\delta \geq \alpha_0:= (2/|V|)^{1/3}$.
It suffices to show that there exists a universal constant $C \geq 1$ such that
\begin{equation}
\label{eq:subpolynomial_claim}
\text{If } \alpha,\delta \geq \alpha_0 \quad \text{ and } \quad \frac{p_c(\alpha,1-\delta)}{p_c(\alpha,\delta)} \geq e^\delta
\qquad \text{ then } \qquad
\delta \leq \sqrt\frac{C \log d}{\log |V|}.
\end{equation}
Fix $\alpha_0\leq \alpha \leq 1$ and $\alpha_0\leq\delta \leq 1/2$ and write $p_0=p_c(\alpha,\delta)$ and $p_1=p_c(\alpha,1-\delta)$. 
If $p_0\leq p \leq p_1$ then $\p_p(\|K_1\|\geq \alpha) (1-\p_p(\|K_1\|\geq \alpha)) \geq \delta(1-\delta) \geq \frac{1}{2}\delta$ and it follows from \cref{thm:KKL_subpolynomial} that there exists a universal constant $c>0$ such that
\begin{align}
1 \geq 1-2\delta&=\int_{p_0}^{p_1} \frac{\dif}{\dif p}\p_p(\|K_1\|\geq \alpha) \dif p \\&\geq \frac{c\delta}{2} \log |V| \int_{p_0}^{p_1}  \left[p(1-p) \log \frac{2}{p(1-p)}\right]^{-1} \dif p,
\end{align}
where we used that every edge has at least $|V|/2$ edges in its $\Aut(G)$ orbit on any vertex-transitive graph.
To estimate this integral we first use the substitution $p=\phi(x):=e^x/(e^x+1)$, which satisfies $\dif p/\dif x= e^x/(e^x+1)^2=p(1-p)$, to write
\begin{equation} \begin{split}
\int_{p_0}^{p_1}\left[p(1-p) \log \frac{2}{p(1-p)}\right]^{-1} \dif p &= \int_{x_0}^{x_1} \left[\log \frac{2(e^x+1)^2}{e^x} \right]^{-1} \dif x \\&\geq \int_{x_0}^{x_1} \frac{1}{|x|+\log 8} \dif x, 
\end{split}\end{equation}
where we write $x_i=\phi^{-1}(p_i)=\log p_i/(1-p_i)$ and use the elementary bound $(e^x+1)^2/e^x =  e^x + 2 + e^{-x} \leq 4e^{|x|}$ in the final inequality.
The logarithmic derivative of $\phi(x)$ is $1/(e^x+1)$ so that if $p_1\geq e^\delta p_0$ then we have that
\[
\int_{x_0}^{x_1} \frac{1}{e^x+1} \dif x \geq \delta \qquad \text{ and hence that } \qquad x_1-x_0 \geq (e^{x_0}+1)\delta.
\]
It follows that if $p_1 \geq e^\delta p_0$ then
\[
\int_{x_0}^{x_0+(e^{x_0}+1)\delta} \frac{1}{|x|+\log 8} \dif x \leq \frac{2}{c\delta \log |V|},
\]
from which the claim may easily be proven via case analysis according to whether $x_0 \leq 0$ or $x_0>0$, noting that $x_0 \geq \phi^{-1}(1/2d) \geq -\log 2d$ since $p_0 \geq 1/2d$. In the first case we use that $|x|+\log 8 \leq |x_0|+\log 8 + \delta \leq |x_0|+3$ for every $x_0 \leq x \leq x_0 + \delta$ to deduce that
\begin{align}
\label{eq:x0negative}
\frac{\delta}{3+\log 2d} \leq \int_{x_0}^{x_0+\delta} \frac{\dif x}{|x_0|+3} &\leq \frac{2 }{c \delta \log |V|} &&\text{ when $x_0\leq 0$},
\intertext{while in the case $x_0>0$ we lower bound the integral by the minimum of the integrand times the length of the interval to obtain that}
\frac{2}{2+\log 8} \delta \leq \frac{(e^{x_0} + 1)\delta}{x_0 + (e^{x_0} + 1)\delta + \log 8 }  &\leq \frac{2}{c\delta \log |V|} &&\text{ when $x_0>0$.}
\label{eq:x0positive}
\end{align}
Putting together \eqref{eq:x0negative} and \eqref{eq:x0positive} completes the proof.
\end{proof}

\begin{proof}[Proof of \cref{thm:quantitative-sup}]
The claim follows immediately from \Cref{thm:2point_main,prop:sharp_giant_subpolynomial,thm:main_single_graph}.
\end{proof}

\section{Non-molecular graphs have the sharp density property} \label{section:the_sharp-giant_property}

In this section we complete the proofs of our main theorems, \cref{thm:main_sparse} and \cref{thm:main}. The most important remaining step is to deduce \cref{thm:main_sparse} and the implication (i) $\Rightarrow $ (ii) of \cref{thm:main} from \Cref{cor:main_sequence} by proving the following proposition.

\begin{prop} \label{prop:molecules:main}
	Let $\mathcal H \subseteq \mathcal F$ be an infinite set. If $\mathcal H$ does not have the sharp density property, then $\mathcal H$ contains a molecular subset.
\end{prop}

\begin{remark}
It follows from \cref{thm:main} and \Cref{cor:main_sequence} that the converse of this proposition also holds, that is, that molecular sequences do not have the sharp density property.
\end{remark}

Recall from \cref{section:deducing_uniqueness} that a sequence of graphs is said to have the sharp density property if the emergence of a giant cluster of a given density always has a sharp threshold. As we saw in \cref{subsec:subpolynomial_sharpness}, it is an immediate consequence of standard sharp threshold results \cite{MR1371123,MR1194785,MR1303654} that this property holds whenever the graphs in question have bounded or subalgebraic vertex degrees. Indeed, these results imply that \emph{any} increasing event depending in a sufficiently symmetric way on $m$ i.i.d.\ Bernoulli-$p$ random variables has a sharp threshold \emph{provided that this threshold occurs around a value of $p$ that is subalgebraically small in $m$;} \cref{lem:pc_lower} implies that the latter condition is satisfied for the event $\{\|K_1\|\geq \alpha\}$ whenever $G_n$ has subalgebraic degrees. Unfortunately it is not true in general that every symmetric increasing event has a sharp threshold without this condition on the location of the threshold. For example, the event that the Erd\H{o}s--R\'enyi graph contains a triangle has a coarse threshold on the scale $p=\Theta(n^{-1})$ and the event that the Erd\H{o}s--R\'enyi graph contains a tetrahedron has a coarse threshold on the scale $p=\Theta(n^{-2/3})$ \cite[Chapter 10.1]{MR3524748}. Thus, to prove \cref{prop:molecules:main} we will need to use specific properties of the event $\{\|K_1\|\geq \alpha\}$ on non-molecular graphs.

\medskip

Our proof will apply a theorem first established by Bourgain \cite{bourgain} and sharpened by Hatami \cite{hatami}, which, roughly speaking, states that any event that does \emph{not} have a sharp threshold must be heavily influenced by the status of a small number of edges. Throughout this section, the prime in e.g.\ $\p_p'$ will always refer to a $p$-derivative. We say that an event $\mathcal A$ is \emph{non-trivial} if $\mathcal A \not= \emptyset $ and $\mathcal A^c \not=\emptyset$. (The following theorem actually only requires that $\mathcal A \not= \emptyset$.)

\begin{thm} [Hatami 2012, Corollary 2.10] \label{thm:hatami}
	Let $G = (V,E)$ be a finite graph, and let $\mathcal A \subseteq \{ 0,1 \}^E$ be a non-trivial increasing event. For every $p \in (0,1/2]$ and $\epsilon > 0$, there is a set of edges $F \subseteq E$ such that $\p_p \bra{ \mathcal A \mid F \subseteq \omega } \geq 1 - \eps$ and
	\[
		\abs{F} \leq \exp\bra{ 10^{13} \left\lceil p \cdot \p_p' \bra{ \mathcal A } \right\rceil^2 \p_p \bra{ \mathcal A}^{-2} \epsilon^{-2}}.
	\]
	In particular, for each $\eps>0$ there exists a constant $C(\eps)<\infty$ such that if $\p_p \bra{ \mathcal A } \geq \eps$ and $p \cdot \p_p'(\cA) \leq \eps^{-1}$ then 
	there is a set of edges $F \subseteq E$ such that $\p_p \bra{ \mathcal A \mid F \subseteq \omega } \geq 1 - \eps$ and $|F|\leq C(\eps)$.
\end{thm}

The relevance of this theorem to sharp-threshold phenomena is made clear by the following elementary lemma, which shows that when $\p_p(\cA)$ does not have a sharp threshold there must be a good supply of parameters where $\p_p(A)$ is not close to $0$ or $1$ and $p \cdot  \p_p'(\cA)$ is not large. Note that if $\cA \subseteq \{0,1\}^E$ is a non-trivial increasing event then its probability $\p_p(\cA)$ is a strictly increasing function of $p\in [0,1]$ by \cite[Theorem 2.38]{MR1707339} and hence defines an invertible increasing homeomorphism $[0,1]\to[0,1]$.

\begin{lem} \label{lem:molecules:finding_parameters}
	Let $G = (V,E)$ be a finite graph, and let $\mathcal A \subseteq \{ 0,1 \}^E$ be a non-trivial increasing event. Define $f:[0,1] \to [0,1]$ by $f(p) := \p_p \bra{ \mathcal A }$.  If $f^{-1} (1 - \delta) \geq (1 + \eps) f^{-1} (\delta)$ for some $\eps \in (0,1]$ and $0<\delta\leq 1/2$ then 
	\[
		\mathcal L \bra{\cubra{ p \in f^{-1} [\delta,1-\delta] : p f'(p) \leq \frac{4}{\eps}} } \geq \frac{1}{2}\mathcal L\left(f^{-1} [\delta,1-\delta]\right),
	\]
	where $\mathcal L$ denotes the Lebesgue measure on $[0,1]$.
\end{lem}

\begin{proof}[Proof of \cref{lem:molecules:finding_parameters}]
Let $I:=f^{-1}[\delta,1-\delta]$. The function $f$ is differentiable, as it is a polynomial, and satisfies
\[
	\int_I p f'(p) \;\mathrm{d}p \leq f^{-1} (1-\delta) \int_I f'(p) \;\mathrm{d} p = (1-2\delta)f^{-1} (1-\delta) \leq f^{-1}(1-\delta).
\]
On the other hand, rearranging our hypothesis $f^{-1} (1 - \delta) \geq (1 + \eps) f^{-1} (\delta)$ gives $f^{-1}(1-\delta) \leq \bra{1 + \frac{1}{\eps}} \sqbra{ f^{-1}(1-\delta) - f^{-1}(\delta) }\leq \frac{2}{\eps}\cL(I)$ and hence that
\[
	\frac{1}{\mathcal L(I)} \int_I p f'(p) \;\mathrm{d}p \leq \frac{2}{\eps}.
\]
The result follows by applying Markov's inequality to the normalised Lebesgue measure on $I$.
\end{proof}

When the edges in the set $F$ given by \cref{thm:hatami} are open, the event $\mathcal A$ occurs with conditional probability at least $1-\eps$. 
Since $\p_p \bra{ \mathcal A \mid F \subseteq \omega} \geq \p_p \bra{ \mathcal A }$ by Harris's inequality, this result is only interesting when $\p_p \bra{ \mathcal A }$ is significantly smaller than $1-\eps$. In this case, the state of $F$ plays a decisive role in determining whether $\mathcal A$ occurs in the sense that the event $\{\omega \notin \mathcal A$ but $\omega \cup F \in \mathcal A \}$ occurs with good probability. This motivates our definition of a subgraph $H$ that \emph{activates} an event $\cA$, a generalisation of being a closed pivotal edge.

\begin{defn}
	Let $G = (V,E)$ be a finite graph, let $H \subseteq G$ be a subgraph, and let $\mathcal A \subseteq \{ 0,1 \}^E$ be an increasing event. We say $H$ \emph{activates} the event $\mathcal A$ in a configuration $\omega$ if $\omega \not\in \mathcal A$ but $\omega \cup H \in \mathcal A$. For every density $\alpha$, we simply say \emph{$H$ activates $\alpha$} to mean $H$ activates the event $\{ \den{K_1} \geq \alpha \}$, and we label this event $\act_\alpha(H)$.
\end{defn}

\begin{cor} \label{cor:molecules:get_subgraph}
	For every $\delta >0$ there exists $\eps > 0$ such that if $G = (V,E)$ is a finite, simple, vertex-transitive graph  and $\alpha, p \in (0,1]$ are such that $\p_p \bra{ \den{K_1} \geq \alpha } \in [\delta,1-\delta]$ and $p \cdot \p_p' \bra{ \den{K_1} \geq \alpha }  \leq \delta^{-1}$ then there exists a subgraph $H$ of $G$ such that $\abs{E(H)} \leq \eps^{-1}$ and $\p_p \bra{ \act_\alpha(H) } \geq \eps$.
\end{cor}

\begin{proof}[Proof of \Cref{cor:molecules:get_subgraph}]
The case $p \leq 1/2$ follows trivially from \cref{thm:hatami}. 
	Now assume $p \geq 1/2$. Since $\{ \den{K_1} \geq \alpha \}$ is invariant under automorphisms of $G$ and $G$ is vertex-transitive we can apply \cref{thm:KKL_subpolynomial} to obtain that there exists a universal constant $c>0$ such that
	\[		
		 \p_p' \bra{ \den{K_1} \geq \alpha }  \geq c \cdot  \p_p \bra{ \den{K_1} \geq \alpha } \p_p \bra{ \den{K_1} < \alpha } \log \abs{V}.
	\]
	 Plugging in our assumed bounds on $p$, $p \cdot  \p_p' \bra{ \den{K_1} \geq \alpha }$, and $\p_p \bra{ \den{K_1} \geq \alpha }$ gives $\abs{V} \leq e^{2/c\delta^3}$, in which case $|E|\leq e^{4/c\delta^3}$ and the result holds trivially with $\eps=\delta \wedge e^{-4/c \delta^3}$.
\end{proof}

When the event $\{ \den{K_1} \geq \alpha \}$ has a \emph{coarse} (i.e., not sharp) threshold on $G = (V,E)$, this lemma gives us a subgraph $H \subseteq G$ that has a good probability of activating $\alpha$. When $H$ has only a single edge $e$, we can use this in the other direction to establish a \emph{lower bound} on the sharpness of the phase transition. Indeed, $e$ activates $\alpha$ if and only if $e$ is closed and pivotal for the event $\{ \den{K_1} \geq \alpha \}$. So, by Russo's formula and \cref{lem:pc_lower} we have that
\[\begin{split}
	p \cdot \p_p' \bra{ \den{K_1} \geq \alpha } & p \sum_{f \in E} \p_p \bra{ f \text{ is pivotal for } \{ \den{K_1} \geq \alpha \}  } \\&\gtrsim \frac{\abs{ \operatorname{Orb(e)} }}{\deg G} \p_p \bra{ \act_\alpha(e) },
\end{split}\]
where $\operatorname{Orb(e)}$ denotes the orbit of the edge $e$ under the action of the automorphism group $\operatorname{Aut} G$. Contrasting this with the assumed \emph{upper bound} on the sharpness of the phase transition with which we started, we can extract information about the underlying graph $G$. For example, we immediately deduce that $G$ is dense, and we are only one step away from concluding that $G$ is part of a molecular sequence. Our main challenge is to reduce to this case, that is to say, to show that we can take $H$ to be a subgraph with a single edge. 

\medskip

During the proof we will want to apply a \emph{sprinkling} argument, where by slightly increasing the parameter $p$ we can make strict subsets of an activator $H$ become activators. 
One difficulty is that as we increase the percolation parameter, we may form a giant cluster with density at least $\beta$, in which case \emph{nothing} can be an activator. As such, we must carefully choose the amount that we sprinkle by at each step. To this end, we will work only with a special sequence of such values constructed by the following lemma.

\medskip

For the remainder of this section, given a finite, connected, vertex-transitive graph $G=(V,E)$, we write $d=\deg (G)$ and $\alpha_0=(2/|V|)^{1/3}$.  
Given $\beta>0$ and $\delta \in (0,1/2]$ we write
  $I = I(G,\beta,\delta)=[ p_c(\beta,\delta) , p_c(\beta,1-\delta) ]$ and $Q = Q(G,\beta,\delta)= \{ p \in I : p \cdot \p_p'(\|K_1\|\geq \beta) \leq \frac{4}{\delta} \}$.

\begin{lem}[A good sequence for sprinkling] \label{lem:molecules:function_and_set}
Let $G=(V,E)$ be a finite, connected, vertex-transitive graph, let $\delta>0$, and let $0<\beta \leq 1$.
	 There exists a sequence $(p_n)\seq$ in $Q$ such that
	\[
		p_{n+1} - p_n \geq 3^{-(n+1)} \mathcal L(Q) \quad \text{and} \quad \p_{p_{n+1}}(\|K_1\|\geq \beta) - \p_{p_{n}}(\|K_1\|\geq \beta) \leq 2^{-n}
	\]
	for every $n\geq 0$.
\end{lem}

\begin{proof} [Proof of \cref{lem:molecules:function_and_set}]
We will prove more generally that if $X \subseteq [0,1]$ is a non-empty closed set and $f:X\to[0,1]$ is an increasing (but not necessarily continuous) function then there exists a sequence $(x_n)\seq$ in $X$ such that 
		$x_{n+1} - x_n \geq 3^{-(n+1)} \mathcal L(X)$ and $f(x_{n+1}) - f(x_n)\leq 2^{-n}$ 
	for every $n\geq 0$. The claim is trivial if $\cL(X)=0$ since we can take $x_n$ constant in this case, so we may assume that $\cL(X)>0$. 
	First consider the case $X = [0,1]$. Let $x_0 := 0$ and define $(x_n)\seq$ recursively as follows. Assume we have defined $x_n$ for some $n \geq 0$. Set $x_{n,i} := x_n + i 3^{-(n+1)}$ for each $i \in \{1,2,3\}$ and define
	\[
		x_{n+1} :=
		\begin{cases}
			x_{n,1} \quad &\text{if } f(x_{n,2})-f(x_{n,1}) \leq f(x_{n,3})-f(x_{n,2}), \\
			x_{n,2} \quad &\text{otherwise}.
		\end{cases}
	\]
	For each $n \geq 0$ write $y_n:=x_{n,3}$ so that $x_n \leq y_m$ for every $n\geq m \geq 0$.
	It follows by induction that $x_n \in [0,1]$, $x_{n+1} - x_n \geq 3^{-(n+1)}$, and $f(y_n)-f(x_n)\leq 2^{-n}$ for every $n\geq 0$, and the claim follows since $f(x_{n+1})\leq f(y_n)$ for every $n\geq 0$.

 	Now let $X \subseteq [0,1]$ be an arbitrary closed set with $\mathcal L(X) > 0$. Since $X$ is closed, we can define an increasing function $\phi : [0,1] \to X$ such that 
	\[
		\mathcal L ( [ 0,\phi(x) ] \cap X ) = x \cL(X)
	\]
	for all $x \in [0,1]$. Construct a sequence $(x_n)\seq$ by the above procedure but for the function $f \circ \phi$ instead of $f$. Then the sequence $(\phi(x_n))\seq$ has the properties that $\phi(x_n) \in X$, $\phi(x_{n+1}) - \phi(x_n) \geq (x_{n+1} - x_n) \mathcal L(X) \geq 3^{-(n+1)} \mathcal L(X)$, and $f ( \phi(x_{n+1})) - f (\phi(x_n)) \leq 2^{-n}$ for every $n \geq 0$ as required.
\end{proof}

We now state our key technical sprinkling proposition.

\begin{prop}[Reducing to a single edge by sprinkling] \label{prop:molecules:shrinking_activator_to_edge}
Let $G=(V,E)$ be a finite, simple, connected, vertex-transitive graph, let $\alpha_0 \leq \delta \leq 1/2$ and $\alpha_0 \leq \alpha \leq \beta \leq 1$, and suppose that $p_c(\beta,1-\delta) > e^\delta p_c(\beta,\delta)$. Let $(p_n)\seq$ be as in \cref{lem:molecules:function_and_set}. 
		 For each $\eps >0$ there exists $N=N(\alpha,\delta,\eps)$ such that for each $n \geq N$ there exists  $\eta_n=\eta_n(\alpha,\delta,\eps)>0$ such that if $H$ is a subgraph of $G$ with $\abs{E(H)} \leq \eps^{-1}$ and 
		$\p_{p_n} (\act_\beta(H) ) \geq \eps$
	then there exists $e\in E(H)$  such that $\p_{p_m}(\act_\beta(e)) \geq \eta_n$ for every $m \geq n + N$.
\end{prop}

Note that the choice of edge $e\in E(H)$ may depend on the choice of $n \geq N$ and that the constants $N$ and $\eta_n$ are independent of $G$ and $\beta$.

\medskip

We now show how \cref{prop:molecules:main} follows from \Cref{cor:molecules:get_subgraph,prop:molecules:shrinking_activator_to_edge}, deferring the proof of \cref{prop:molecules:shrinking_activator_to_edge} to \cref{subsec:sprinkling}.

\begin{proof}[Proof of \cref{prop:molecules:main} given \cref{prop:molecules:shrinking_activator_to_edge}]
Let $G=(V,E)$ be a finite, simple, connected, vertex-transitive graph of degree $d$ and let $\alpha_0=(2/|V|)^{1/3}$. Let $\alpha_0 \leq \alpha \leq \beta \leq 1$, let $\alpha_0 \leq \delta \leq 1/2$,  and suppose that
\[
	p_c(\beta,1-\delta) > e^\delta p_c(\beta,\delta).
\]
It suffices to prove that there exist positive constants $c=c(\alpha,\delta)$ and $C=C(\alpha,\delta)$ such that if $|V|\geq C$ then $\deg G \geq c |V|$ and there exists an automorphism-invariant set of edges $F$ with $|F|\leq C |V|$ such $G\setminus F$ has at most $C$ connected components.

\medskip

Let $(p_n)\seq$ be as in \cref{lem:molecules:function_and_set}.
 It follows from \Cref{cor:molecules:get_subgraph} that there exists $\eta_1=\eta_1(\delta)$ such that for each $p\in Q$ there exists a subgraph $H$ of $G$ such that $|E(H)|\leq \eta_1^{-1}$ and $\p_p(\act_\beta(H)) \geq \eta_1$. Applying this fact with $p=p_{n}$ for an appropriately large constant $n$, it follows from \cref{prop:molecules:shrinking_activator_to_edge} that there exist positive constants $\eta_2=\eta_2(\alpha,\delta)$ and $N_1 =N_1(\alpha,\delta)$ and an edge $e_0 \in E$ such that 
\[\p_{p_{n}}(\act_\beta(e_0))\geq \eta_2\]
for every $n \geq N_1$.
Since the edge $e_0$ activates $\beta$ if and only if $e_0$ is closed and pivotal for the event $\{ \den{K_1} \geq \beta \}$, we have by Russo's formula that
\[\begin{split}
	\p_{p_{n}}'(\|K_1\|\geq \beta) &= \sum_{e \in E} \p_{p_{n}} \bra{ e \text{ is pivotal for } \{ \den{K_1} \geq \beta \} } \\&\geq \abs{\operatorname{Orb}(e_0)} \p_{p_{n}} \bra{ \act_\beta( e_0  ) } \geq \eta_2 \abs{\operatorname{Orb}(e_0)}
\end{split}\]
for every $n\geq N_1$.
Since $p_{N_1} \geq p_c(\beta,\delta)$ and $\alpha,\delta \geq \alpha_0$ it follows from \cref{lem:pc_lower} that
 $p_{N_1}\geq 1/2d$ and hence that
\[
	\eta_2 \frac{\abs{\operatorname{Orb}(e_0)}}{2d} \leq p_{N_1} \cdot \p_{p_{N_1}}'(\|K_1\|\geq \beta) \leq \frac{4}{\delta}
\]
for every $n \geq N_1$, 
where the upper bound follows since $p_{n}\in Q$ for every $n\geq 1$.
Since $G$ is vertex-transitive $\abs{\operatorname{Orb}(e_0)} \geq \frac{1}{2}\abs{V}$ and it follows that the constant $c_2=c_2(\alpha,\delta)=\delta \eta_2/16$ satisfies
\[
d \geq c_2 |V|.
\]
 This establishes the desired density of $G$. 
  
  \medskip

  All that remains is to find an $\operatorname{Aut} G$-invariant linear-sized set of edges $F \subseteq E$ that disconnects $G$ into a bounded number of components. Let $C_1=6/c_2$, let $\eta_3 = \eta_2 \wedge C_1^{-1}$, and let  $N_2=N(\alpha,\delta,\eta_3)$ and $\eta_4=\eta_{N_1 \vee N_2}(\alpha,\delta,\eta_3)$ be as in \cref{prop:molecules:shrinking_activator_to_edge}. We claim that if we set $k = N_1 \vee N_2 + N_2 $ then there exists a constant $C_2 = C_2(\alpha,\delta)$ such that the set
\[
	F := \{ e \in E : \p_{p_k} \bra{\act_\beta( e )} \geq \eta_4 \}
\]
has the desired properties when $|V|\geq C_2$. 
This set $F$ is clearly $\operatorname{Aut}G$-invariant. Moreover, by Russo's formula, since $p_k \cdot \p_{p_k}'(\|K_1\|\geq \beta) \leq \frac{4}{\delta}$ and $p_k \geq 1/2d \geq 1/2|V|$, we have that
\[\begin{split}
	 \abs{F} &\leq \frac{1}{\eta_4}\sum_{e \in E} \p_{p_k} \bra{ e \text{ is pivotal for } \{ \den{K_1} \geq \beta \} }\\ &=  \frac{1}{\eta_4}\p_{p_k}'(\|K_1\|\geq \beta) \leq \frac{8}{\delta \eta_4} |V|,
\end{split}\]
so that $|F|$ is at most linear in $|V|$. 
 Since $|F| \leq \frac{8}{\delta \eta_4}|V|$ and $d=\deg G \geq c_2 |V|$, we have that 
\[
\deg(G\setminus F) \geq c_2 |V| - \frac{16}{\delta \eta_4}.  
\]
Thus, there exists a constant $C_2=C_2(\alpha,\delta)$ such that if $|V|\geq C_2$ then 
\[
\deg(G\setminus F) \geq \frac{c_2}{2} |V|.
\]
It now suffices to prove that $G \setminus F$ is not connected when $|V|\geq C_2$. Indeed, once this is shown it follows automatically that $G\setminus F$ has a bounded number of components since if $G \setminus F$ has $m$ components then $\frac{c_2}{4} |V|^2 \leq |E| \leq m (|V|/m)^2 + |F|$, hence there exists a constant $C_4=C_4(\alpha,\delta)$ such that $m \leq 1+4/c_2$ when $|V|\geq C_4$.

\medskip

Suppose for contradiction that $|V|\geq C_2$ and that $G \setminus F$ is connected.
Since $G\setminus F$ is connected and $\deg(G\setminus F) \geq \frac{c_2}{2} |V|$, it follows from \cref{lem:dense-diam} that $\operatorname{diam}(G\setminus F) \leq 6/c_2 = C_1$.
Let $P$ be a path of length at most $C_1$ connecting the endpoints of $e_0$ in $G \setminus F$. If $e_0$ activates $\beta$ then so does the set $P$, and it follows that 
\[
|P| \leq C_1 \leq \eta_3^{-1} \quad \text{ and } \quad \p_{p_{N_1 \vee N_2}}(\act_\beta(P)) \geq \p_{p_{N_1 \vee N_2}}(\act_\beta(e_0)) \geq \eta_2 \geq \eta_3.
\]
As such, it follows from the definitions of the quantities $N_2=N(\alpha,\delta,\eta_3)$ and $\eta_4=\eta_{N_1 \vee N_2}(\alpha,\delta,\eta_3)$ as in \cref{prop:molecules:shrinking_activator_to_edge} that there exists an edge $e_1 \in P$ such that  $\p_{p_k}(\act_\beta(e_1)) \geq \eta_4$. This implies that $e_1\in F$, a contradiction. \qedhere

\end{proof}

\subsection{Proof of \Cref{prop:molecules:shrinking_activator_to_edge}}
\label{subsec:sprinkling}

In this section we complete the proof of \cref{prop:molecules:main} by proving \cref{prop:molecules:shrinking_activator_to_edge}. The proof will proceed inductively, showing that -- by changing to a different value of $p\in Q$ if necessary -- we can reduce the number of edges in the subgraph $H$ given by \Cref{cor:molecules:get_subgraph} while keeping $\p_p \bra{ \act_\alpha(H) }$ bounded away from zero. 
More precisely, we will deduce \cref{prop:molecules:shrinking_activator_to_edge} as an inductive consequence of the following lemma.

\begin{lem}[Removing one edge by sprinkling] \label{lem:molecules:shrinking_activator}
Let $G=(V,E)$ be a finite, simple, connected, vertex-transitive graph, let $\alpha_0 \leq \delta \leq 1/2$ and $\alpha_0 \leq \alpha \leq \beta \leq 1$, and suppose that $p_c(\beta,1-\delta) > e^\delta p_c(\beta,\delta)$. Let $(p_n)\seq$ be as in \cref{lem:molecules:function_and_set}. 
		 For each $\eps >0$ there exists $N=N(\alpha,\delta,\eps)$ such that if $n \geq N$ then there exists $\eta_n=\eta_n(\alpha,\delta,\eps)>0$ such that if $H$ is a subgraph of $G$ with $\abs{E(H)} \leq \eps^{-1}$ and 
		$\p_{p_n} (\act_\beta(H) ) \geq \eps$
	then there exists a subgraph $H'$ of $H$ with $|E(H')|\leq \max\{1,|E(H)|-1\}$  such that $\p_{p_m}(\act_\beta(H')) \geq \eta_n$ for every $m > n$.
\end{lem}

Note that the choice of subgraph $H'$ of $H$ may depend on the choice of $n \geq N$ and that the constants $N$ and $\eta_n$ are independent of $G$ and $\beta$. Also, note that $N$ and $\eta_n$ here are not the same $N$ and $\eta_n$ as in the statemenet of \cref{prop:molecules:shrinking_activator_to_edge}.

\begin{proof}[Proof of \cref{prop:molecules:shrinking_activator_to_edge} given \cref{lem:molecules:shrinking_activator}]
Applying \cref{lem:molecules:shrinking_activator} iteratively $\lfloor 1/\eps\rfloor$ times yields the claim.
\end{proof}

For the remainder of this section we fix a finite, simple, connected, vertex-transitive graph $G=(V,E)$ of degree $d$, fix $\alpha_0 \leq \delta \leq 1/2$ and $\alpha_0 \leq \alpha \leq \beta \leq 1$ such that $p_c(\beta,1-\delta) > e^\delta p_c(\beta,\delta)$, and let $(p_n)\seq$ be as in \cref{lem:molecules:function_and_set}. We also continue to write $I = I(G,\beta,\delta)=[ p_c(\beta,\delta) , p_c(\beta,1-\delta) ]$ and $Q = Q(G,\beta,\delta)= \{ p \in I : p \cdot \p_p'(\|K_1\|\geq \beta) \leq \frac{4}{\delta} \}$.

\medskip

When $G$ has bounded vertex degrees, and hence critical parameters are bounded away from zero, \cref{lem:molecules:shrinking_activator} could be proven easily by (classical) insertion-tolerance. The problem is rather more delicate in general. The idea is to use vertex-transitivity of $G$ to find many copies of $H$ that each activate $\beta$ simultaneously, then \emph{sprinkle}, i.e.\! open a small number of additional edges by slightly increasing the percolation parameter, and argue that, after sprinkling, many of the copies of $H$ have strict subgraphs that are activators.  To ensure that sprinkling reduces the number of edges necessary to activate $\beta$ in a positive proportion of the copies of $H$, we need these copies to be well-connected to each other in the open subgraph. We guarantee this by sticking a large cluster to each endpoint of an edge in $H$, using the next lemma.

\begin{lem} \label{lem:molecules:attaching_clusters}
For every $\eps > 0$ there exists $\eta=\eta(\alpha,\delta,\eps)>0$ such that if $p\in Q$ and $H$ is a subgraph of $G$ with $\abs{E(H)} \leq \frac{1}{\eps}$ and $\p_p \bra{ \act_\beta(H) } \geq \eps$ then there is an edge $e \in E(H)$ with endpoints $u$ and $v$ such that
	\[
		\p_p \bra{ \act_\beta(H) \cap \{ \den{K_u} \geq \eta \} \cap \{ \abs{K_v} \geq \eta d \} } \geq \eta.
	\]
\end{lem}

The proof of this lemma uses the quantitative insertion-tolerance estimate of \cref{prop:insertion-tolerance} together with the following theorem of the second author \cite[Theorem 2.2]{uni-tightness}, which guarantees that the size of the largest intersection of a cluster with a fixed set of vertices is always of the same order as its mean with high probability. We state a special case of the theorem that is adequate for our purposes.

\begin{thm}[Universal Tightness] \label{thm:molecules:uni_tightness} There exist universal constants $C,c>0$ such that the following holds.
	Let $G = (V,E)$ be a countable, locally finite graph, let $\Lambda \subseteq V$ be a finite non-empty set of vertices, and let $p \in [0,1]$ be a parameter. Set $\abs{M} := \max\{ \abs{K_v \cap \Lambda} : v \in V  \}$. Then 
	\[
		\p_p \bra{ \abs{M} \geq \alpha \e_p \abs{M} } \leq Ce^{-c\alpha} \qquad \text{and} \qquad \p_p \bra{ \abs{M} \leq \eps \e_p\abs{M} } \leq C \eps
	\]
	for every $\alpha \geq 1$ and $0 < \eps \leq 1$.
\end{thm}

\begin{proof}[Proof of \Cref{lem:molecules:attaching_clusters}]
  We may assume that $\eps \leq \delta$.
	 We may also assume that $H$ has no isolated points, so that we have the bound $\abs{V(H)} \leq 2\abs{E(H)} \leq \frac{2}{\eps}$. When $H$ activates $\beta$ we must have that $\den{\bigcup_{u \in V(H)} K_u} \geq \beta$ and hence by the pigeonhole principle that there exists $u \in V(H)$ such that \[\den{K_u} \geq \frac{\beta}{\abs{V(H)}} \geq \frac{\alpha\eps}{2}.\] It follows that there exists a \emph{fixed} vertex $u \in V(H)$ such that
	\begin{equation}
		\p_p \bra{ \act_\beta(H) \cap \cubra{ \den{K_u} \geq \frac{\alpha\eps}{2}} } \geq \frac{1}{|V(H)|}\p_p \bra{ \act_\beta(H) } \geq   \frac{\eps}{\abs{V(H)}} \geq \frac{\eps^2}{2}.
	\end{equation}

	Since $H$ has no isolated points, $u$ is the endpoint of some edge $e \in E(H)$. Let $v$ be the other endpoint of $e$, let $N$ be the set of neighbours of $v$ in $G$, and let $X$ be an $\omega$-connected subset of $N$ (i.e.\! a subset of $N$ that is contained in a single $\omega$-cluster) of maximum size. Since $p\geq p_c(\alpha,\delta)$ and $G$ is vertex-transitive,
	\begin{equation}
		\e_p \abs{X} \geq \e_p \left[ \abs{K_1 \cap N} \right] = \abs{N} \e_p \den{K_1} \geq \alpha \delta d.
	\end{equation}
	 Applying \Cref{thm:molecules:uni_tightness} it follows that there exists a positive constant $c_1=c_1(\alpha,\delta,\eps)$ such that $\p_p \bra{ \abs{X} \leq c_1 d} \leq \frac{\eps^2}{4}$ and hence by a union bound that
	\begin{equation}
	\label{eq:Xlarge}
		\p_p \bra{ \act_{\beta}(H) \cap \cubra{ \den{K_u} \geq \frac{\alpha\eps}{2}} \cap \cubra{ \abs{X} \geq c_1 d } } \geq \frac{\eps^2}{4}.
	\end{equation}
	
	\medskip

	To obtain a similar bound with $\abs{K_v}$ in place of $\abs{X}$, we use insertion tolerance to open an edge connecting $v$ to $X$, which forces $\abs{K_v} \geq \abs{X}$. Unfortunately we cannot argue this way directly since opening this edge may produce a cluster with density at least $\beta$, in which case $\act_{\beta}(H)$ would no longer hold. To avoid this issue, we first claim that there exists a constant $C_1=C_1(\alpha,\delta,\eps)$ such that if $|V|\geq C_1$ then
	\begin{equation} \label{eq:molecules:attaching_clusters:not_act}
		\p_p \bra{ \act_{\beta}(H) \cap \cubra{ \den{K_u} \geq \frac{\alpha\eps}{2}} \cap \cubra{ \abs{X} \geq c_1 d } \cap \act_{\beta}(vX)^c } \geq \frac{\eps^2}{8},
	\end{equation}
	where $vX$ denotes the set of edges with one endpoint equal to $v$ and the other in $X$.
	(Note that since $X$ is $\omega$-connected, each edge in $vX$ individually activates $\beta$ if and only if the entire set $vX$ activates $\beta$.)
	Indeed, suppose that \eqref{eq:molecules:attaching_clusters:not_act} does not hold. We have by \eqref{eq:Xlarge} and a union bound that 
	\[\p_p\bra{ \cubra{ \abs{X} \geq c_1 d} \cap \act_{\beta}(vX)} \geq \frac{\eps^2}{8}\] and hence that
	\[
		\sum_{e \in E: \; e \ni v} \p_p \bra{ \act_{\beta}(e) } \geq \e_p \left[\abs{X} \mathbbm 1_{ \act_{\beta}(vX) }\right] \geq \frac{c_1 \eps^2}{8} d.
	\]
	Applying Russo's formula and using that $G$ is vertex-transitive, it follows that
	\[
		 \p_p' \bra{ \den{K_1} \geq \beta } \geq \frac{c_1 \eps^2}{16} d\abs{V} .
	\] 
	 Since $p \geq p_c(\alpha_0,\alpha_0)$ we also have that $p\geq 1/2d$ by \cref{lem:pc_lower} and hence that
	\[
		p \cdot \p_p' \bra{ \den{K_1} \geq \beta } \geq \frac{c_1 \eps^2}{32} \abs{V}.
	\] 
	Since $\eps \leq \delta$, this contradicts the hypothesis that $p\in Q$ whenever $\abs{V} \geq C_1:=128/(c_1 \eps^3)$, completing the proof of the claim.

\medskip

	 There are finitely many graphs with $\abs{V} \leq C_1$, hence the lemma holds trivially in this case and we may assume without loss of generality that \eqref{eq:molecules:attaching_clusters:not_act} holds. Since we have that $p \geq 1/2d$ as above, we can use insertion tolerance \cref{prop:insertion-tolerance} to open an edge in $vX$ in the event appearing on the left hand side of \eqref{eq:molecules:attaching_clusters:not_act}, giving that
	\[
		\p_p \bra{ \act_{\beta}(H) \cap \cubra{ \den{K_u} \geq \frac{\alpha\eps}{2}} \cap \cubra{ \abs{K_v} \geq c_1 d } } \geq \eta,
	\]
 for some $\eta=\eta(\alpha,\delta,\eps)>0$ as claimed.
\end{proof}

Once we have many copies of $H$ that activate $\beta$ and have large clusters stuck to the copies of $u$ and $v$, we use the following easy fact about equivalence relations to deduce that the copies of $e$ are well-connected to each other. More precisely, we will use this lemma to show that we can find many large disjoint sets of copies of the edge $e$ in which any two copies $\gamma_1(e)$ and $\gamma_2(e)$ of $e$ are $\omega$-connected by paths $\gamma_1(u) \leftrightarrow \gamma_2(u)$ and $\gamma_1(v) \leftrightarrow \gamma_2(v)$.

\begin{lem} \label{lem:molecules:sets}
	Let $X$ be a non-empty finite set, let $\sim$ be an equivalence relation on $X$, and let $Y \subset X$ be a non-empty subset of $X$. For each $x\in X$ write $[x]$ for the equivalence class of $x$ under $\sim$. 
	  If $\min_{y \in Y} \abs{ [y] } \geq \frac{2\abs{X}}{\abs{Y}}$ then there exists a collection $( Z_i )_{i \in I}$ of disjoint subsets of $Y$ such that each $Z_i$ is contained in an equivalence class of $\sim$ and that satisfies the inequalities
	\[
		\abs{Z_i} \geq \frac{\abs{Y}}{2\abs{X}} \min_{y \in Y} \abs{ [y] } \qquad \text{and} \qquad \abs{I} \geq \frac{\abs{X}}{4 \min_{y \in Y} \abs{ [y] } }.
	\]
\end{lem}

\begin{proof}[Proof of \cref{lem:molecules:sets}]
	Write $m=\min_{y \in Y} \abs{ [y] }$ and define 
	\[
		Y_- := \cubra{ y \in Y : \abs{ [y] \cap Y } \leq \frac{\abs{Y}}{2\abs{X}} \abs{ [y] }  } 
		\qquad \text{ and } \qquad Y_+ := Y \setminus Y_-.
	\]
	Observe that if $\sC$ denotes the set of equivalence classes of $\sim$ then
	\[
	\frac{2|X|}{|Y|}|Y_-| \leq \sum_{y\in Y} \frac{|[y]|}{|[y]\cap Y|} = \sum_{C \in \sC} |C| \mathbbm{1}(C \cap Y \neq \emptyset) \leq |X|,
	\]
	so that $|Y_-|\leq \frac{1}{2}|Y|$ and $\abs{Y_+} \geq \frac{1}{2}\abs{Y}$.
	Let $(Z_i)_{i \in I}$ be a maximal collection of disjoint subsets of $Y_+$ such that each $Z_i$ is contained in an equivalence class of $\sim$ and $\abs{Z_i} = \left\lceil \frac{\abs{Y}}{2\abs{X}} m \right\rceil$ for each $i \in I$. Every element $y \in Y_+$ has
	\[
		\abs{[y] \cap Y} \geq \left \lceil \frac{\abs{Y}}{2\abs{X}} \abs{ [y] }\right\rceil \geq \left\lceil\frac{\abs{Y}}{2\abs{X}} m\right\rceil.
	\]
	So, by maximality, the union $\bigcup_{i \in I} Z_i$ contains at least half the elements in $[y] \cap Y$ for every $y \in Y_+$. So the union $\bigcup_{i \in I} Z_i$ contains at least half the elements in $Y_+$, and we deduce that
	\[
		\abs{I} \geq \frac{\abs{Y}/4}{ \left\lceil \frac{\abs{Y}}{2\abs{X}} m \right\rceil } \geq \frac{\abs{X}}{4m},
	\]
	where we used the hypothesis $m \geq \frac{2\abs{X}}{\abs{Y}}$ in the final inequality.
\end{proof}

We are now ready to complete the proof of  \cref{lem:molecules:shrinking_activator}.
The final step to reduce the number of edges in $H$ is to open a small number of these well-connected copies of $e$ by slightly increasing the percolation parameter. When $H$ activates $\beta$ and $u \leftrightarrow v$, $H\setminus \{ e\}$ also activates $\beta$. Since the copies of $e$ are well-connected, opening this small number of edges is actually sufficient to ensure a positive proportion of the copies of $H \setminus \{ e \}$ activate $\beta$. By linearity of expectation, we conclude that at this higher parameter the set $H \setminus \{ e \}$ activates $\beta$ with good probability, as required.

\begin{proof}[Proof of \cref{lem:molecules:shrinking_activator}]
Since $\alpha,\delta \geq \alpha_0$ we have by \Cref{lem:molecules:finding_parameters} and \Cref{lem:pc_lower} that
\[\mathcal L(Q) \geq \frac{1}{2}{\mathcal L(I)} = \frac{1}{2}(p_c(\beta,1-\delta) - p_c(\beta,\delta)) \geq \frac{1}{2}(e^\delta-1) p_c(\beta,\delta) \geq \frac{\delta}{4d},\]
where we used that $e^\delta-1 \geq \delta$ in the final inequality.
Fix $\eps>0$ and $n \geq 1$, and suppose that $H$ is a finite subgraph of $G$ with $|E(H)|\leq \eps^{-1}$ such that $\p_{p_n} \bra{ \act_\beta(H) } \geq \eps$. By \Cref{lem:molecules:attaching_clusters}, we can find $\eps_1=\eps_1(\alpha,\delta,\eps)$ and an edge $e \in E(H)$ with endpoints $u$ and $v$ such that the event
\[
	\mathcal A := \act_\beta(H) \cap \{ \den{K_u} \geq \eps_1 \} \cap \{ \abs{K_v} \geq \eps_1 d \} \qquad \text{ satisfies $\p_{p_n} \bra{\mathcal A} \geq \eps_1$.}
\]
For each $x \in V$, pick an automorphism $\phi_x \in \operatorname{Aut} G$ such that $\phi_x(v) = x$, so that $\phi_x(u)$ is a neighbour of $x$ for each $x\in V$. These maps exist because $G$ is vertex-transitive. Define $X := \{ x \in V : \phi_x^{-1}(\omega) \in \mathcal A \}$. We have by linearity of expectation that $\e_{p_n} \den{X} = \p_{p_n} \bra{ \mathcal A } \geq \eps_1$ and hence by Markov's inequality that $\p_{p_n} ( \den{X} \geq \frac{1}{2}\eps_1 ) \geq \frac{1}{2}\eps_1$. 
We will now prove that we can take $N=\lceil \log_2(8/\eps_1)\rceil$, so assume from now on that $n\geq \log_2(8/\eps_1)$.

\medskip

Write $\mathcal B := \{ \den{X} \geq \frac{1}{2}\eps_1 \}$ and consider an arbitrary configuration $\omega \in \mathcal B$. Every $\phi_x(u)$ with $x \in X$ has $\|K_{\phi_x(u)}\| \geq \eps_1$. So, by the pigeonhole principle, we can find a subset $Y \subseteq X$ with $\den{Y} \geq \frac{1}{2}\eps_1^2$ such that $\{ \phi_x(u) : x \in Y \}$ is contained in a single cluster of $\omega$. By definition of $X$ and $\cA$, every vertex $y \in Y$ has $ \abs{K_y} \geq \eps_1 d$.

\medskip

We next claim that there exists a constant $C_1=C_1(\alpha,\delta,\eps)$ such that if $|V|\geq C_1$ then $\eps_1 d |Y| \geq 4|V|$. Indeed, since $\den{Y}\geq \frac{1}{2}\eps_1^2$, if this inequality does not hold then we must have that $d\leq 8 \eps_1^{-3}$. 
In this case, since $p_n\in Q$ and $p_n\geq 1/2d$, and since every edge has $\operatorname{Aut}G$-orbit of size at least $\abs{V}/2$ by transitivity, it follows from \cref{thm:KKL_subpolynomial} that there exists a constant $c_1=c_1(\alpha,\delta,\eps)$ such that
\[
\frac{4}{\delta} \geq p_n \cdot \p_{p_n}'(\|K_1\|\geq \beta) \geq c_1 \log |V|,
\]
which rearranges to give the claim. Since graphs with $|V| < C_1$ can be handled trivially, we may assume throughout the rest of the proof that $|V| \geq C_1$  and hence that $\eps_1 d |Y|\geq 4|V|$.

\medskip

  By construction, every vertex $y \in Y$ has $\abs{K_y} \geq \eps_1 d$. So by splitting clusters, we can find an equivalence relation that is a refinement of $\xleftrightarrow{\omega}$ in which the equivalence class of each vertex $y \in Y$ has between $\eps_1 d/2$ and $\eps_1 d$ total vertices. We now apply \Cref{lem:molecules:sets} to this equivalence relation with the sets $V$ and $Y$ in place of $X$ and $Y$. (The hypothesis of \Cref{lem:molecules:sets} is met because $\eps_1 d \abs{Y} \geq 4 \abs{V}$.) This yields a collection of disjoint $\omega$-connected subsets $(Z_r)_{r \in R}$ of $Y$ such that for every $r$,
\[
	\abs{Z_r} \geq \frac{\eps_1 d\abs{Y}}{4 \abs{V}}\geq \frac{\eps_1^3d}{8}  \qquad \text{and} \qquad \abs{R} \geq \frac{\abs{V}}{4 \eps_1 d} \geq \frac{\abs{V}}{4 d}.
\]
Whenever $H$ activates $\beta$ and $u \leftrightarrow v$, $H \setminus \{ e \}$ also activates $\beta$. On the event $\mathcal B$ we must have that $x \leftrightarrow y$ and $\phi_x(u) \leftrightarrow \phi_{y}(u)$ whenever $x,y$ both belong to $Z_r$ for some $r\in R$. Thus, on the event $\mathcal B$, if there exist  $r \in R$ and $x \in Z_r$ such that $x \leftrightarrow \phi_{x}(u)$ then $\phi_{x'}(H \setminus \{ uv \})$ activates $\beta$ for \emph{every} $x' \in Z_r$. 

\medskip

In view of this fact, our next step will be to increase the percolation parameter to open an edge $\phi_x(e)$ with $x \in Z_r$ for a positive proportion of the indices $r \in R$, thus making a positive proportion of the copies $\phi_{x} (H \setminus \{ e \})$ with $x \in \bigcup_{r \in R} Z_r$ activate $\beta$.

\medskip

Let $m>n \geq \lceil \log_2(8/\eps_1)\rceil$, and let $\p$ be the joint law of the standard monotone coupling $(\omega, \omega')$ of percolation with the two parameters $p_n\leq p_{m}$.
It suffices to prove that there exists a constant $\eta_n=\eta_n(\alpha,\delta,\eps) > 0$ such that
\begin{equation}
\label{eq:etan_target}
\p_{p_m}(\act_\beta(H\setminus\{e\})) \geq \eta_n.
\end{equation}
Recall that the increasing sequence $(p_n)\seq$ was defined so that
\[
p_{m}-p_n \geq 3^{-(n+1)} \cL(Q) \geq 3^{-(n+1)} \frac{\delta}{4d}\] and \[\p_{p_{m}}(\|K_1\|\geq \beta)-\p_{p_{n}}(\|K_1\|\geq \beta) \leq 2^{-n+1}
\]
for every $m>n$. The assumption that $m>n \geq \log_2(8/\eps_1)$ implies that $\p_{p_{m}}(\|K_1\|\geq \beta)-\p_{p_{n}}(\|K_1\|\geq \beta) \leq \frac{\eps_1}{4}$, and since $\p_{p_n} \bra{ \mathcal B } \geq \frac{\eps_1}{2}$ a union bound gives that 
\begin{align}
	\p \bra{ \omega \in \mathcal B  \text{ and }  \|K_1 (\omega')\| < \beta } &\geq \p_{p_n} \bra{ \mathcal B } - \p \bra{ \|K_1(\omega)\| < \beta \leq \|K_1 (\omega')\|  } \nonumber
	\\&
	=\p_{p_n} \bra{ \mathcal B } - \\&\quad\left(\p_{p_m} \bra{ \|K_1(\omega)\| \geq \beta}-\p_{p_n} \bra{ \|K_1(\omega)\| \geq \beta}\right)\nonumber
	\\&\geq \frac{\eps_1}{2} - \frac{\eps_1}{4} = \frac{\eps_1}{4},
\end{align}
and hence in particular that
\begin{equation} \label{eq:molecules:induction:avoid_giant}
	\p \bra{ \|K_1 (\omega')\| < \beta \mid \omega \in \mathcal B } \geq \frac{\eps_1}{4}.
\end{equation}

\medskip

Consider an arbitrary configuration $\xi \in \mathcal B$ and let $(Z_r)_{r \in R}$ be a collection of sets as defined via \cref{lem:molecules:sets} above, which we take to be a function of $\xi$. (Note in particular that the index set $R$ depends on $\xi$.) For each  $r \in R$ and $x \in Z_r$ we have that
\[
	\p \bra{ \phi_x(e) \in \omega' \mid \omega = \xi } \geq \frac{p_m - p_n}{1-p_n} \geq 3^{-(n+1)} \frac{\delta}{4d},
\]
 and since $\abs{ \{ \phi_x(e) : e \in Z_r \} } \geq \frac{1}{2}\abs{Z_r} \geq \frac{\eps_1^3}{16} d$ (where the factor of $1/2$ accounts for the distinction between oriented and unoriented edges, with $\phi_x$ acting bijectively on the former), we deduce that there exists $\eps_2=\eps_2(\alpha,\delta,\eps)>0$ such that
\begin{align}
	\p \bra{ \exists x \in Z_r : \; \phi_x(e) \in \omega' \mid \omega = \xi } &\geq 1 - \bra{ 1 - 3^{-(n+1)}\frac{\delta}{4d} }^{\eps_1^3d/16} \nonumber
	\\
	&\geq 1-\exp\left[-3^{-(n+1)} \frac{\eps_1^3\delta}{64}\right] \geq \eps_2 3^{-n},
\end{align}
where we used the inequality $1-x\leq e^{-x}$ in the second line. Condition on the event $\omega = \xi$ and consider the random set $J := \{ r \in R : Z_r \xleftrightarrow{\omega'} Y \}$. Note that for all $r_1,r_2 \in R$ with $Z_{r_1}\not\xleftrightarrow{\omega} Y$ and $Z_{r_2} \not\xleftrightarrow{\omega} Y$, we have $\{\phi_x(e) : x \in Z_{r_1}\} \cap \{\phi_x(e) : x \in Z_{r_2}\} = \emptyset$. So by independence, it follows that there exists a constant $C_2=C_2(\alpha,\delta,\eps)$ such that if $|R|\geq C_2 3^n$ then
\begin{equation} \label{eq:molecules:induction:help_lots}
	\p \bra{ \abs{J} \geq \frac{\eps_2}{2} 3^{-n}\abs{R} \bigm| \omega = \xi } \geq 1 - \frac{\eps_1}{8}.
\end{equation}

\medskip

We will now proceed by case analysis according to whether $|R(\xi)|\geq C_2 3^n$ for every $\xi \in \cB$ or $|R(\xi)|<C_2 3^n$ for some $\xi \in \cB$.
First assume that $\abs{R(\xi)} \geq C_2 3^n$ for every $\xi \in \mathcal B$. On the event ${\omega \in \mathcal B}$, we write $(Z_r)_{r \in R}$ and $J$ for the corresponding sets defined \emph{with respect to $\omega$}. We have by \eqref{eq:molecules:induction:avoid_giant} and \eqref{eq:molecules:induction:help_lots} that
\begin{equation} \label{eq:molecules:induction:helped__and_no_giant}
	\p \bra{ \abs{J} \geq \frac{\eps_2}{2}3^{-n} \abs{R}  \text{ and }  \| K_1( \omega' )\| < \beta \bigm| \omega \in \mathcal B } \geq \frac{\eps_1}{4} - \frac{\eps_1}{8} = \frac{\eps_1}{8},
\end{equation}
and hence that
\begin{equation} \label{eq:molecules:induction:helped__and_no_giant2}
	\p \bra{ \omega \in \mathcal B,\, \abs{J} \geq \frac{\eps_2}{2}3^{-n} \abs{R}, \text{ and } \| K_1( \omega' )\| < \beta } \geq  \frac{\eps_1^2}{16}.
\end{equation}
Consider a pair of configurations $(\omega,\omega')$ satisfying the event whose probability is estimated in \eqref{eq:molecules:induction:helped__and_no_giant2} and
 pick $r \in J$ and $y \in Z_r$. Since $\phi_{y}(H)$ activates $\beta$ in $\omega$ and $\|K_1(\omega')\| < \beta$, we also have that $\phi_{y}(H)$ activates $\beta$ in $\omega'$. By definition of $J$, we can find $x \in Z_r$ such that $\phi_{x}(e)$ is open in $\omega'$. 
Since $x,y \in Z_r$ we have by definition that $x$ and $y$ are connected in $\omega$, and since $x,y\in Y$ we have by definition of $Y$ that $\phi_x(u)$ and $\phi_y(u)$ are also connected in $\omega$.
  Since $\phi_x(e)$ is open in $\omega'$ and $\phi_x(v) = x$, we deduce that $x$, $y$, $\phi_x(u)$, and $\phi_y(u)$ all belong to the same cluster of $\omega'$.
Since $\phi_y(H)$ activates $\beta$ in $\omega'$ and  $y$ is connected to $\phi_y(u)$ in $\omega'$, we deduce that $\phi_y(H \setminus \{ e\})$ activates $\beta$ in $\omega'$ also. Since this holds for every $r \in J$ and $y \in Z_r$, we have
\[
	|W|:=\abs{ \cubra{ x \in X : \omega' \in \act_{\beta}( \phi_x(H \setminus \{ e \})  }} \geq \abs{J} \min_{r \in J} \abs{Z_r}.
\]
 Since $\abs{J} \geq \frac{\eps_2}{2} 3^{-n}\abs{R} \geq 3^{-n} \frac{\eps_2}{8} \frac{\abs{V}}{d}$ and every $Z_r$ has $\abs{Z_r} \geq \frac{\eps_1^3}{8} d$, we have that $\den{W} \geq \frac{\eps_1^3 \eps_2}{64} 3^{-n}$ and hence that
\[
	\e \den{W} \geq \frac{\eps_1^5 \eps_2}{2^{10}} 3^{-n}.
\]
Thus, it follows by vertex-transitivity that $\p_{p_m} \bra{ \act_{\beta}(H \setminus \{ e \} ) } \geq \frac{\eps_1^5 \eps_2}{2^{10}} 3^{-n}$. This bound has the required form, completing the proof of \eqref{eq:etan_target} in this case.

\medskip

It remains to consider the case that  $\abs{R(\xi)} < C_2 3^n$ for some $\xi \in \cB$. Since we always have $\abs{R} \geq \abs{V}/ 4 d$, we must have in this case that
\[
	|V| \leq 4 C_2 3^n d.
\]
Observe that if $\omega \in \mathcal B$ and $\den{ K_1 ( \omega' ) } < \beta$ then $\omega' \in \mathcal B$ also. Thus, by \eqref{eq:molecules:induction:avoid_giant} and the fact that $\p_{p_n} \bra{ \mathcal B } \geq \frac{\eps_1}{2}$, we have that
\[
	\p_{p_m} \bra{ \mathcal B } \geq \frac{\eps_1^2}{8}.
\]
On the event $\mathcal B$, pick one of the sets $Z_r$ that are guaranteed to exist by our earlier argument, call it $Z$, and let $U := \{ x \in Z : \; \omega \in \act_\beta(\phi_x(e)) \}$. If
\[
	\p_{p_m} \Bigl( \mathcal B \cap \bigl\{ \den{U} \geq \frac{1}{2} \den{Z} \bigr\}\Bigr) \geq \frac{1}{2} \p_{p_m} \bra{ \mathcal B}
\]
then we have by vertex-transitivity and the bound $|Z|\geq \eps_1^3 d/8 $ that
\[
	\p_{p_m} \bra{ \act_\beta( e ) } \geq \e_{p_m} \sqbra{ \den{U} \mathbbm{1}(\cB) } \geq \frac{\eps_1^3 d}{16 |V|} \cdot \frac{\eps_1^2}{16} \geq \frac{\eps_1^5}{2^{10} C_2} \cdot 3^{-n}
\]
as required.
Conversely, if
\[
	\p_{p_m} \Bigl( \mathcal B \cap \bigl\{ \den{U} \geq \frac{1}{2} \den{Z} \bigr\}\Bigr) \leq \frac{1}{2} \p_{p_m} \bra{ \mathcal B},
\]
then we use insertion-tolerance (\cref{prop:insertion-tolerance} - using the fact that $\den{Z} \geq \frac{\eps_1^3}{64 C_2 3^n}$) to open a single edge in $\{ \phi_{x}(e) : x \in Z \setminus U \}$ on the event $\mathcal B \cap \{ \den{U} \leq \frac{1}{2} \den{Z}\}$. This does not create a cluster with density at least $\beta$, since none of the edges $\phi_x(e)$ with $x \in Z \setminus U$ activates $\beta$. Arguing as before, after opening this edge, \emph{every} set $\phi_x(H \setminus \{ e \})$ with $x \in Z$ activates $\beta$. Thus, it follows by the same vertex-transitivity argument as above that there exists a positive constant $\eps_3(n)=\eps_3(\alpha,\delta,\eps,n)$ such that
\[
	\p_{p_m} \bra{ \act_\beta( H \setminus \{ e\} ) } \geq \eps_3.
\]
This completes the proof.
\end{proof}

\subsection{Completing the proof of the main theorems}
\label{subsec:Bollobas}

In this section we complete the proof of \Cref{thm:main} and hence of \Cref{thm:main_sparse}.

\begin{proof}[Proof of \cref{thm:main}]
The implication (i) $\Rightarrow$ (ii) follows immediately from \Cref{cor:main_sequence} and \cref{prop:molecules:main}, while the implication (iii) $\Rightarrow$ (iv) is trivial. As such, it remains only to prove (ii) $\Rightarrow$ (iii) and (iv) $\Rightarrow$ (i).
 
\medskip

We begin with the implication (ii) $\Rightarrow$ (iii), i.e., the claim that molecular sets admit linear $1/3$-separators. To see this, note that if $\mathcal H$ is $m$-molecular for some $m\geq 2$ then there exists a constant $C$ such that for each $G \in \mathcal H$ there is an automorphism-invariant set of edges $F$ such that $|F|\leq C|V(G)|$ and $F$ disconnects $G$ into $m \geq 2$ connected components each of size $|V(G)|/m$. If we take $A$ to be the union of $\lceil m/2 \rceil$ of these components then $|V(G)|/3 \leq |A| = \lceil m/2 \rceil |V(G)|/m \leq 2|V(G)|/3$ and $|\partial_E A|\leq |F|\leq C|V(G)|$ so that $\mathcal H$ has linear $1/3$-separators as claimed.

\medskip

We next prove the implication (iv) $\Rightarrow$ (i). We will use the following theorem \cite{dense} identifying the location of the critical threshold for (not necessarily transitive) dense graph sequences.

\begin{thm} [Bollob\'{a}s, Borgs, Chayes, Riordan 2010] \label{lem:dense-thresh}
	Let $(G_n)$ be a dense sequence of finite, simple graphs with $\abs{V(G_n)} \to \infty$, and for each $n\geq 1$ let $\lambda_n$ be the largest eigenvalue of the adjacency matrix of $G_n$. For each $c > 0$
  \[
    \lim_{n\to\infty} \p_{\lambda_n^{-1}}^{G_n}(\|K_1\|\geq c) =0,
  \]
  and for each $\eps > 0$ there exists $\delta > 0$ such that
  \[
    \lim_{n\to\infty} \p_{(1+\eps)\lambda_n^{-1}}^{G_n}(\|K_1\|\geq \delta) =1.
  \]
\end{thm}

Note that if $A$ is the adjacency matrix of a finite graph then $A$ is self-adjoint and the largest eigenvalue (a.k.a.\ the Perron--Frobenius eigenvalue) of $A$ coincides with the $L^2$ operator norm of $A$. This norm satisfies $\|A\| \geq \langle A \mathbbm{1},\mathbbm{1}\rangle / \langle \mathbbm{1},\mathbbm{1}\rangle = 2 |E|/|V|$ where $\mathbbm{1}$ is the constant-one function, hence 
 if $(G_n)\seq$ is dense then the sequence of largest eigenvalues $\lambda_n$ satisfies \[\liminf |V(G_n)|^{-1}\lambda_n >0.\] \medskip

Suppose that $\mathcal H$ is dense and admits linear $\theta$-separators for some $\theta\in (0,1/2]$, so that there exists a constant $C_1$ such that for each $G \in \mathcal H$ there exists a set $A(G) \subseteq V(G)$ with $\theta |V(G)|\leq|A(G)|\leq (1-\theta)|V(G)|$ and $|\partial_E A(G)|\leq C_1|V(G)|$. For each $G \in \mathcal H$ let $H(G)$ and $H(G)^c$ denote the subgraphs of $G$ induced by $A(G)$ and $A(G)^c$ respectively. Since every vertex of $G$ has degree $2|E(G)|/|V(G)|$ we have that
\[
\min\{|E(H(G))|,|E(H(G)^c)|\} \geq \theta|E(G)|-|\partial_E A(G)|.
\]
For large $n$ we have that $\theta|E(G)|\gg |\partial_E A(G)|$ and hence that $(H(G))_{G \in \mathcal H}$ and $(H(G)^c)_{G \in \mathcal H}$ are both dense. Thus, it follows from \cref{lem:dense-thresh} that there exists a constant $C_2$ such that if we set $p(G)=C_2/|V(G)|$ for each $G \in \mathcal H$ then both $H(G)$ and $H(G)^c$ contain a giant component with high probability under percolation with parameter $p(G)$. The same also holds at $q(G)=2C_2/|V(G)|$, which is $\eps$-supercritical for $\mathcal H$ for an appropriate choice of $\eps>0$. But at this same parameter $q(G)$ the expected number of edges connecting $A(G)$ and $A(G)^c$ is bounded by $2C_1C_2$, so that by Poisson approximation the probability that there are no such edges is bounded away from zero uniformly over $G \in \mathcal H$. Thus the probability that $q(G)$-percolation on $G$ contains at least two giant components is bounded away from zero uniformly over $G \in \mathcal H$, and hence $\mathcal H$ does not have the supercritical uniqueness property. It follows that the supercritical uniqueness property fails for any set $\mathcal H'$ with $\mathcal H \subseteq \mathcal H' \subseteq \mathcal F$, by extending $q$ from $\mathcal H$ to $\mathcal H'$ by setting $q(G) := 1$ for all $ G \in \mathcal H' \backslash \mathcal H$.
\end{proof}

\begin{remark}
\label{rem:dense}
The proof of (iv) $\Rightarrow$ (i) above shows more generally that any sequence of finite, simple graphs with linear minimal degree and linear $\theta$-separators for some $\theta \in (0,1/2]$ fails to have the supercritical uniqueness property. On the other hand, the results of \cite{dense} imply that dense graphs without \emph{subquadratic} separators have the supercritical uniqueness property. It is natural to wonder in light of \cref{thm:main} whether the failure of supercritical uniqueness in graphs of linear minimal degree is always characterized by the existence of linear separators, without the assumption of vertex-transitivity.

\begin{figure}
\centering
\includegraphics[width=\textwidth]{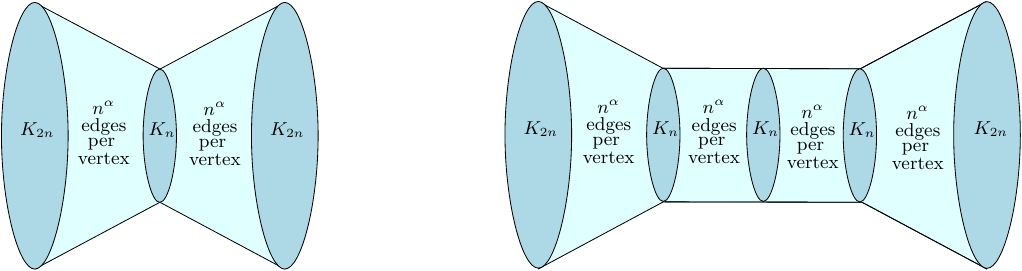}
\caption{Schematic illustration of the graphs discussed in \cref{rem:dense}}
\label{fig:dense_counterexample}
\end{figure}

This is not the case. Indeed, let $0<\alpha<1/2$ and suppose that we take two copies of $K_{2n}$ and one copy of $K_n$ arranged in a line with the two copies of $K_{2n}$ at the end and the copy of $K_n$ in the middle. We may glue these copies together in such a way that each vertex is connected to each of the complete graphs adjacent to its own complete graph by between $n^\alpha$ and $3 n^\alpha$ edges. It is easily verified that the smallest separators in the resulting graph sequence are of order $n^{1+\alpha}$ and that there will exist two distinct giant clusters with high probability when $p=3/4n$. We focus on the second claim, which is more involved. For such $p$ the two copies of $K_{2n}$ are supercritical and each contains a giant cluster, while the copy of $K_{n}$ is subcritical and has largest cluster of order $\log n$ with high probability. Thus, when we add in the edges between the various complete graphs, the probability that there exists a cluster in the copy of $K_n$ that has an edge connecting it to both of its neighbouring copies of $K_{2n}$ is small: a $K_n$-cluster of size $m =O(\log n)$ has both such edges adjacent to it with probability of order $m^2 n^{2\alpha-2}$, and since there are at most $2n$ such clusters the total conditional probability is $O((\log n)^2 n^{2\alpha-1})=o(1)$ with high probability, yielding the claim. This gives an example of a linear minimal-degree graph sequence that does not have linear separators but does not have the supercritical uniqueness property either.  By considering longer chains of copies of $K_n$ connecting the two copies of $K_{2n}$ as in \cref{fig:dense_counterexample}, one can obtain similar examples where the minimal size of a separator scales like an arbitrary power of $n$ between $n$ and~$n^2$.
\end{remark}

\section{Closing remarks} \label{section:things_to_find_a_home_for}

\label{subsec:counterexamples}

\noindent \textbf{Counterexamples.}
We now discuss examples demonstrating that \cref{thm:main} does not extend to arbitrary insertion-tolerant percolation models on the torus or to critical percolation.

\begin{example}[Multiple giants at criticality]
The cycle $\Z/n\Z$ (with its standard generating set) has multiple giant components with good probability when $p=1-\lambda/n$, and the set of closed edges converges to a Poisson process on the circle as $n\to \infty$. As observed by Alon, Benjamini, and Stacey \cite{percolation-expanders}, by considering the highly asymmetric torus $T_n := (\Z/2^n \Z) \times (\Z/n\Z)$ (with its standard generating set) one can obtain similar behaviour at values of $p$ that are bounded away from $1$.

Since these authors did not include a proof, let us now very briefly indicate how the analysis of this example works. We assume for notational simplicity that $n$ is a power of $2$ and hence is a factor of $2^n$. Let $X$ be the set of cylinders in $(\Z/2^n \Z) \times (\Z/n\Z)$ of the form $[kn,(k+1)n]\times (\Z/n\Z)$ whose vertices are incident to some simple cycle of dual edges that wrap around the torus. When $p>1/2$, it follows by sharpness of the phase transition on $\Z^2$ that there exists a constant $c_p>0$ with $c_p \to \infty$ as $p \to 1$ such that each particular cylinder has probability at most $e^{-c_pn}$ to belong to $X$, this probability being bounded by the probability that a box of size $n$ in $\Z^2$ intersects a dual cluster of diameter at least $n$. On the other hand, since the correlation length on $\mathbb Z^2$ diverges as $p \uparrow 1/2$, there exists a fixed $p > 1/2$ such that $\e_{p} \abs{X} \to \infty$ as $n \to \infty$. Using this one can prove that if we define $p_n$ to be the unique value such that $\e_{p_n} |X|=1$ then $\liminf p_n > 1/2$ and $\limsup p_n < 1$. It is fairly straightforward to prove that there are multiple giant clusters with positive probability at $p_n$. Indeed, using the same exponential decay estimates one can prove that non-neighbouring cylinders are highly de-correlated, so that $\e_{p_n}|X|^2=O(1)$ and there is a good probability for $X$ to contain at least two elements that are well-spaced around the torus. On this event there must be at least two giant components with high probability: since $p_n$ is bounded away from $p_c(\Z^2)=1/2$,  each vertex of the torus has good probability to belong to a cluster that wraps around the torus, and any two such vertices can be disconnected only if there are two closed dual cut-cycles separating them. The events that two distant vertices belong to such wrapping clusters are highly de-correlated, so that the number of vertices belonging to wrapping clusters is linear with high probability and the claim follows.
\end{example}

\begin{example}[Insertion-tolerance on the torus is not enough]
Consider the symmetric torus $(\Z/10n\Z)^2$ with its standard generating set. Consider the model defined as follows: First,  select a pair of vertical strips of the form $[k,k+n]\times (\Z/10n\Z)$ and $[k+5n,k+6n]\times (\Z/10n\Z)$ uniformly from among the $n$ available possibilities. Declare each edge belonging to one of these strips open with probability $1/4$ and each edge not belonging to one of these strips open with probability $3/4$. Using standard properties of subcritical and supercritical percolation on $\Z^2$, one easily obtains that this model contains exactly two giant clusters with high probability: the two large high-density strips will each contain a giant cluster with high probability, while there will be no clusters crossing either of the two thin low-density strips with high probability. By also applying a random rotation in $\{0,\pi/2,\pi,3\pi/4\}$ one obtains an automorphism invariant, uniformly insertion-tolerant, percolation-in-random-environment model on  $(\Z/10n\Z)^2$ with the same properties. As such, one should not expect any uniqueness-of-the-giant-component results to hold on finite graphs at anywhere near the same generality as found in the Burton--Keane theorem \cite{burton-keane}, even when restricting to symmetric tori of fixed dimension. By taking the relative width of the low-density strips to go to zero in a well-chosen manner as $n\to\infty$, one can construct a similar example in which the number of giant clusters is either one or two each with good probability and any two vertices are connected with good probability.
\end{example}

\label{subsec:existence}

\noindent \textbf{The supercritical existence property}. 
\cref{thm:main} can be though of a geometric characterisation of the infinite sets $\mathcal H \subseteq \mathcal F$ for which supercritical percolation has \emph{at most} one giant cluster with high probability. 
We now briefly address the complementary problem of whether there is \emph{at least} one giant cluster with high probability in supercritical percolation, noting that the definitions only ensure that such a cluster exists with \emph{good} probability (i.e., with probability bounded away from zero).

Let $\cH \subseteq \mathcal F$ be an infinite set. We say that $\cH$ has the \emph{supercritical existence property} if for every supercritical assignment $p : \cH \to [0,1]$ there exists a constant $\alpha > 0$ such that
\[
	\liminf_{G \in \cH} \p_{p(G)}^{G} \bra{ \text{the largest cluster contains at least } \alpha \abs{V(G)} \text{ vertices} } = 1.
\]
Notice that the sharp density property immediately implies the supercritical existence property. However the converse is false because (as we will show below) molecular graphs also have the supercritical existence property. This might lead one to suspect that the supercritical existence property always holds. In fact, this is not the case, and the counterexamples can once again be exactly characterised in terms of molecular graphs.
Let us note that the weaker supercritical existence property 
\[
	\lim_{\alpha \downarrow 0} \liminf_{ G \in \cH } \p_{p(G)}^{G} \bra{ \text{the largest cluster contains at least } \alpha \abs{V(G)} \text{ vertices} } = 1
\]
always holds (even without transitivity) as an immediate consequence of the universal tightness theorem \cite{uni-tightness}.
The following is a fairly straightforward consequence of our results and those of \cite{dense}.

\begin{cor} \label{cor:sep}
An infinite set $\cH \subseteq \cF$ has the supercritical existence property if and only if it is not the case that there exist arbitrarily large integers $m$ for which $\cH$ contains an $m$-molecular subset.
\end{cor}

\begin{proof}[Sketch of proof]
Applying the results of \cite{dense} as in the proof of \cref{thm:main} easily yields that if $\cH$ is $m$-molecular then 
\[
	\liminf_{G \in \cH} \p_{ (1-\eps)^{-1}  p_c^{G}( \alpha, \eps ) } ^{G} \bra{ \den{K_1} \leq 1/m } >0 
\]
for every $0<\alpha,\eps <1$, yielding the forward implication of the claim.
We now suppose $\mathcal H$ does not have the supercritical existence property and argue that it contains $m$-molecular subsequences for arbitrarily large values of $m$. 
It is an immediate consequence of \cref{prop:molecules:main} that $\mathcal H$ has at least one molecular subsequence. Suppose for contradiction that the supremal value of $m$ such that $\mathcal H$ has an $m$-molecular subsequence is finite, and denote this supremum by $M$. Since $\mathcal H$ does not have the supercritical existence property, there exists an infinite subset $\mathcal S \subseteq \mathcal H$ such that
\[
	\limsup_{G \in \mathcal S} \p_{ (1-\eps)^{-1}  p_c^{G}( \eps, \eps ) } ^{G} \bra{ \den{K_1} \geq \frac{\eps^2}{4M} } < 1,
\]
and applying \cref{prop:molecules:main} as before we may assume that this subset is $m$-molecular for some $2\leq m \leq M$. By taking a further infinite subset and changing $m$ if necessary, we may assume that this subset has density $\deg (G)/|V(G)|$ converging to some constant $c>0$ and 
does not have a further subset that is $k$-molecular for any $k>m$.

By definition there exists a constant $C$ such that for each $G \in \mathcal S$ there exists an automorphism-invariant set $F_G \subseteq E(G)$ with 
$|F_G|\leq C |V_G|$
 such that $G \setminus F_G$ has $m$ connected components. For each $G \in \mathcal S$, let $H_{G}$ be a graph isomorphic to each of the $m$ connected components of 
$G \backslash F_{G}$. Since $\mathcal S$ does not admit a subset that is $k$-molecular for any $k>m$,
$\mathcal A := \{ H_G : G \in \mathcal S \}$
 cannot itself contain a molecular subset. Thus, it follows from \cref{prop:molecules:main,thm:main} that $\mathcal A$ has the sharp density property and the supercritical uniqueness property. On the other hand, the vertex degrees of $H_G$ with $G \in \mathcal S$ are asymptotically equal to those of $G$ in the sense that the ratio tends to $1$, and \cref{lem:dense-thresh} allows us to compute the location of the percolation thresholds for these transitive dense graph sequences in terms of their vertex degrees. (Recall that the largest eigenvalue for the adjacency matrix of a regular graph is equal to its vertex degree.) So any assignment $\tilde p : \mathcal A \to [0,1]$ built from the assignment $p: \mathcal S \to [0,1]$ with $p(G) := (1-\eps)^{-1/2} \cdot p_c^G(\eps,\eps)$ by arbitrarily picking $\tilde p(H) \in \{ p(G) : H_G = H \}$ for each $H \in \mathcal A$ is itself supercritical for $\mathcal A$. We deduce from \cref{thm:main} that
\[
	\lim_{G \in \mathcal S} \e_{p(G)} ^{H_G} \den{K_2} =0.
\]
 In order for a vertex to belong to the largest cluster in some $G \in \mathcal S$, either it must belong to the largest cluster in its copy of $H_G$, or this is not the case and there is an open edge of $F_G$ incident to its cluster. By vertex transitivity every vertex of $G$ is incident to at most $C$ edges of $F_G$. It follows that, writing $K_i$ for the $i^{\text{th}}$ largest cluster,
\[\begin{split}
\E_{p(G)}^{G}\|K_1\| &\leq  \E_{p(G)}^{H_G}\left[\|K_1\| + \sum_{i\geq 2} Cp(G) \|K_i\| \cdot |K_i| \right] \\&\leq  \E_{p(G)}^{H_G}\|K_1\| + C  p(G) |V(G)| \e_{p(G)}^{H(G)}\|K_2\|.
\end{split}\]
Since $p(G)$ is of order $|V(G)|^{-1}$ the second term tends to zero as $G \to \infty$ with $G \in \mathcal S$, and we deduce that $\e_{p(G)}^{H_G}\|K_1\| \geq \frac{1}{2} \E_{p(G)}^{G}\|K_1\| \geq \frac{\eps^2}{2}$ for all but finitely many $G \in \mathcal S$. Define $p' : \mathcal S \to [0,1]$ by \[p'(G) := (1-\eps)^{-1/2} \cdot p(G) = (1-\eps)^{-1} \cdot p_c^G(\eps,\eps).\] Since $\mathcal A$ does not have any molecular subsets, it follows by Markov's inequality and \cref{prop:molecules:main} that 
\[
	\liminf_{G \in \mathcal S} \p_{  p'(G)}^{G} \bra{ \den{K_1} \geq \frac{\eps^2}{4M} } \geq 
	\liminf_{G \in \mathcal S } \p_{ p'(G)}^{H_G} \bra{ \den{K_1} \geq \frac{\eps^2}{4} }  = 1,
\]
a contradiction. \qedhere
\end{proof}

\bibliographystyle{abbrv}
\bibliography{references}
\end{document}